\documentclass[11pt,reqno]{amsart}
\usepackage{amssymb,amsmath,amscd,amsthm,amsfonts,wasysym,mathrsfs,amsxtra}
\usepackage{spectralsequences}
\usepackage[linktocpage=true]{hyperref}
\hypersetup{colorlinks,linkcolor=blue,urlcolor=blue,citecolor=blue}
\usepackage{mathtools}
\usepackage{marginnote}
\usepackage{appendix}
\usepackage{enumitem}
\usepackage{hyperref}
\usepackage{tikz}
\usepackage[all,cmtip]{xy}
\usepackage{tikz-cd}

\newcommand{\s}{\mathbb{S}}

\newcommand{\CC}{\mathbb{C}}

\newcommand{\NN}{\mathbb{N}}
\newcommand{\PP}{\mathbb{P}}

\newcommand{\GG}{\mathbb{G}}

\newcommand{\Aa}{\mathcal{A}}
\newcommand{\Bb}{\mathcal{B}}
\newcommand{\Cc}{\mathcal{C}}
\newcommand{\Ee}{\mathcal{E}}

\newcommand{\cg}{\mathfrak{g}}

\newcommand{\Dd}{\mathcal{D}}
\newcommand{\Ff}{\mathcal{F}}
\newcommand{\Hh}{\mathcal{H}}
\newcommand{\Kk}{\mathcal{K}}
\newcommand{\Nn}{\mathcal{N}}
\newcommand{\Tt}{\mathcal{T}}
\newcommand{\Pp}{\mathcal{P}}
\newcommand{\Oo}{\mathcal{O}}
\newcommand{\Uu}{\mathcal{U}}
\newcommand{\MM}{\mathcal{M}}
\newcommand{\Qq}{\mathcal{Q}}
\newcommand{\Xx}{\mathcal{X}}

\newcommand{\R}{\mathcal{R}}
\newcommand{\Ss}{\mathcal{S}}
\newcommand{\V}{\mathcal{V}}

\newcommand{\ZZ}{\mathbb{Z}}

\newcommand{\LL}{\mathscr{L}}
\newcommand{\TTt}{\sf{T}}
\newcommand{\SSs}{\sf{S}}
\newcommand{\E}{\sf{E}}
\newcommand{\F}{\sf{F}}
\newcommand{\Hhh}{\sf{H}}

\newcommand{\spi}{\sf{\Psi}}
\newcommand{\SOD}{\sf{SOD}}
\newcommand{\SL}{\mathfrak{sl}}

\newcommand{\kk}{\underline{k}}
\newcommand{\Ll}{\underline{l}}
\newcommand{\bo}{\boldsymbol{1}}
\newcommand{\bS}{\boldsymbol{S}}
\newcommand{\bT}{\boldsymbol{T}}
\newcommand{\brmU}{\boldsymbol{\Rm{U}}}
\newcommand{\brmT}{\boldsymbol{\Rm{T}}}
\newcommand{\brmX}{\boldsymbol{\Rm{X}}}
\newcommand{\tdim}{\mathrm{dim}}

\newcommand{\Hom}{\mathrm{Hom}}
\newcommand{\End}{\mathrm{End}}

\newcommand{\tdet}{\mathrm{det}}
\newcommand{\Gr}{\mathrm{Gr}}
\newcommand{\Fl}{\mathrm{Fl}}
\newcommand{\Coh}{\mathrm{Coh}}
\newcommand{\Rm}{\mathrm}

\newcommand{\GLL}{\mathrm{GL}}
\newcommand{\SLL}{\mathrm{SL}}

\newcommand{\blam} {\boldsymbol{\lambda}}
\newcommand{\bmu} {\boldsymbol{\mu}}

\newtheorem{theorem}{Theorem}[section]
\newtheorem{thmx}{Theorem}

\newtheorem{lemma}[theorem]{Lemma}

\newtheorem{proposition}[theorem]{Proposition}

\newtheorem{corollary}[theorem]{Corollary}
\newtheorem{cory}{Corollary}
\theoremstyle{definition}
\newtheorem{definition}[theorem]{Definition}

\newtheorem{example}[theorem]{Example}

\theoremstyle{remark}
\newtheorem{remark}[theorem]{Remark}

\numberwithin{equation}{section}

\title[Categorical idempotents via shifted 0-affine algebras]{Categorical idempotents via shifted 0-affine algebras}

\begin{document}
	
\emergencystretch 3em

\begin{abstract}
We show that a categorical action of shifted 0-affine algebra naturally gives two families of pairs of complementary idempotents in the triangulated monoidal category of triangulated endofunctors for each weight category. Consequently, we obtain two families of pairs of complementary idempotents in the triangulated monoidal category $\Dd^b\Coh(G/P \times G/P)$.

As an application, this provides examples where the projection functors of a semiorthogonal decomposition are kernel functors, and we determine the generators of the component categories in the Grassmannians case.
\end{abstract}
	
\address{Academia Sinica} \email{youhunghsu@gate.sinica.edu.tw}

\author[You-Hung Hsu]{You-Hung Hsu}
	
\keywords{Derived category, categorification, Demazure operators, categorical idempotents}
	
\makeatletter
\@namedef{subjclassname@2020}{%
	\textup{2020} Mathematics Subject Classification}
\makeatother
	
\subjclass[2020]{Primary 14M15, 18N25, 20C08 : Secondary 18F30, 18G80 }
	
\maketitle

\setcounter{tocdepth}{1}






\section{Introduction}


\subsection{Complementary idempotents} \label{1.1}
Let $A$ be a $\CC$-algebra with unit $1_{A}$. An element $a \in A$ is called an idempotent if $a^2=a$. An idempotent $a \in A$ naturally induces another idempotent $1_{A}-a$, called its complementary idempotent. In particular, they are mutually orthogonal and give a decomposition of $1_{A}$.

Hogancamp \cite{Ho} develops a categorical analog of the above theory of idempotent in the following sense: Given objects $\bT$, $\bS$ in a triangulated monoidal category $\Aa$, we call $(\bT,\bS)$ a \textit{pair of complementary idempotents} if they satisfy mutual orthogonality and give a decomposition of the monoidal identity $\bo_{\Aa}$, see Definition \ref{compleidem}.

In particular, if $\Aa$ acts on a triangulated category $\MM$ by triangulated endofunctors (categorical analog of $A$-modules), then the pair of complementary idempotents $(\bT,\bS)$ gives a short semiorthogonal decomposition $\MM=\langle \text{Im}\bS,\text{Im}\bT \rangle$, where $\text{Im}{\F}$ denotes the minimal full triangulated subcategory generated by the class of objects that are essential images of a functor ${\F}$.

\subsection{Categorical actions}

Given a complex semisimple or Kac-Moody Lie algebra $\cg$. There has been much progress in the categorification of $\cg$ (or the quantum group $\brmU_{q}(\cg)$) and their representations over the past decade. The notion of lifting their representations from vector spaces to categories is called the \textit{categorical actions}.  

For example, when $\cg=\SL_2$, a categorical $\SL_2$-action consists of weight categories $\Cc(\lambda)$ and functors ${\E}:\Cc(\lambda)\rightarrow \Cc(\lambda+2)$, ${\F}:\Cc(\lambda)\rightarrow \Cc(\lambda-2)$ such that 
\begin{equation} \label{eq} 
	{\E\F}|_{\Cc(\lambda)} \cong {\F\E}|_{\Cc(\lambda)} \bigoplus \text{Id}_{\Cc(\lambda)}^{\oplus \lambda} \ \text{if} \ \lambda \geq 0,
\end{equation} a similar equivalence exists for $\lambda \leq 0$.

Note that we do not impose any natural transforms in (\ref{eq}). Understanding the natural transformations between the functors in (\ref{eq}) is important in higher representation theory. Chuang-Rouquier \cite{CR} gives one answer to such a problem for $\SL_{2}$ and later generalized to (simply-laced) Kac-Moody algebras $\cg$ by Khovanov-Lauda \cite{KL1} and Rouquier \cite{R}.

\subsection{Main result}



In \cite{Hsu}, the author defines an algebra called the shifted 0-affine algebra, denoted by $\Uu=\dot{\brmU}_{0,N}(L\SL_{n})$, and defines the notion of  \textit{categorical $\Uu$-action} on certain triangulated 2-categories, see Definition \ref{shited0affine} and Definition \ref{catshifted0}. It is proved that there is a categorical $\Uu$-action on the bounded derived categories of coherent sheaves on $n$-step partial flag varieties, see Proposition \ref{Thm catact}.

In this article, we use categorical $\Uu$-action to construct pairs of complementary idempotents in the triangulated monoidal category of triangulated endofunctors of each weight category. We first state the main result of this article.

\begin{thmx}[Theorem \ref{Theorem 5}, Corollary \ref{mainthmgeo}] \label{mainresult}
Given a categorical $\Uu$-action on $\Kk$. Then for $1 \leq i \leq n-1$, there are pairs of complementary idempotents $({\TTt}'_{i}\bo_{\kk},  {\SSs}'_{i}\bo_{\kk})$ in the triangulated monoidal category $\Hom(\Kk(\kk),\Kk(\kk))$. Moreover, relations between those functors are given.

In particular, when $\Kk(\kk)=\Dd^b\Coh(\Fl_{\kk}(\CC^N))$ in which we prove there is a categorical $\Uu$-action in \cite{Hsu}, we obtain pairs of complementary idempotents that given by Fourier-Mukai kernels $({\Tt}'_{i}\bo_{\kk}, \ {\Ss}'_{i}\bo_{\kk})$ in the triangulated monoidal category $\Dd^b\Coh(\Fl_{\kk}(\CC^N) \times \Fl_{\kk}(\CC^N))$ for $1 \leq i \leq n-1$, where the convolution of kernels gives the monoidal structure.
\end{thmx}

Now we give a brief explanation of the main result. Denote by $\kk=(k_1,...,k_n) \vDash N$ a composition of $N \geq 2$ of $n$ parts. Each $\kk$ corresponds to a $n$-step partial flag variety below
\begin{equation} \label{eq nfl}
\Fl_{\kk}(\CC^N)\coloneqq \{V_{\bullet}=(0=V_{0} \subset V_1 \subset ... \subset V_{n}=\CC^N) \ | \ \tdim V_{i}/V_{i-1}=k_{i} \ \text{for} \ \text{all} \ i\}.	
\end{equation} 

We denote $\Dd^b\Coh(X)$ to be the bounded derived category of coherent sheaves on a variety $X$. The generators  $e_{i,r}1_{\kk}$ and $f_{i,s}1_{\kk}$ with $-k_i-1 \leq r \leq 0$ and $0 \leq s \leq k_{i+1}+1$ of $\Uu$ act on $\bigoplus_{\kk' \vDash N} \Dd^b\Coh(\Fl_{\kk'}(\CC^N))$ by the following correspondence diagram
\begin{equation*}
\xymatrix{ 
&&W^{1}_{i}(\kk)
\ar[ld]_{p_1} \ar[rd]^{p_2}   \\
& \Fl_{\kk}(\CC^N) && \Fl_{\kk+\alpha_{i}}(\CC^N)
}
\end{equation*}  where 
$W^{1}_{i}(\kk) \coloneqq \{(V_{\bullet},V_{\bullet}') \in  \Fl_{\kk}(\CC^N) \times \Fl_{\kk+\alpha_i}(\CC^N) \ | \ V_{j}=V'_{j} \ \Rm{for} \ j \neq i,  \ V'_{i} \subset V_{i}\}$. Here $\alpha_{i}=(0...0,-1,1,0...0)$ is the simple root with the $-1$ is in the $i$th position, and $p_{1}$, $p_{2}$ are the natural projections.  

The generators $e_{i,r}1_{\kk}$ act on $\bigoplus_{\kk' \vDash N} \Dd^b\Coh(\Fl_{\kk'}(\CC^N))$ by lifting to the following functors
\begin{equation*}
{\E}_{i,r}\bo_{\kk} \coloneqq p_{2*}(p_{1}^{*} \otimes (\V_{i}/\V'_{i})^{r}):\Dd^b\Coh(\Fl_{\kk}(\CC^N)) \rightarrow \Dd^b\Coh(\Fl_{\kk+\alpha_{i}}(\CC^N))
\end{equation*} where we denote $\V_{i}$, $\V'_{i}$ to be the tautological bundles on $W^{1}_{i}(\kk)$ of rank $k_{1}+..+k_{i}$, $k_{1}+..+k_{i}-1$ respectively, and thus we form the quotient line bundle $\V_{i}/\V'_{i}$. Similarly for the lift of $f_{i,s}1_{\kk}$ to ${\F}_{i,s}\bo_{\kk}$ in the opposite direction. Furthermore, ${\E}_{i,r}\bo{\kk}$ and ${\F}_{i,s}\bo{\kk}$ are Fourier-Mukai transforms with kernels denoted by $\Ee_{i,r}\bo_{\kk}$ and $\Ff_{i,s}\bo_{\kk}$ respectively, see Proposition \ref{Thm catact} for details.

The idea comes from the observation that the categorical Demazure operators can be expressed as generators in the categorical $\dot{\Uu}_{0,N}(L\SL_{N})$-action and generalized to the $\dot{\Uu}_{0,N}(L\SL_{n})$-action with $n<N$. In more detail, note that we have the full flag variety $\Fl_{(1,1,...,1)}(\CC^N)=G/B$ and the partial flag varieties are $G/P_{i}=\Fl_{(1,1,...,1) + \alpha_{i}}(\CC^N)=\Fl_{(1,1,...,1) - \alpha_{i}}(\CC^N)$ where $G=\SLL_N$, $B \subset G$ is the standard Borel and $P_{i} \subset G$ is the minimal parabolic subgroup corresponds to the simple root $\alpha_i$.

There are natural projection $\pi_{i}:G/B \rightarrow G/P_i$ that induces derived pullback and derived pushforward $ \Dd^b\Coh(G/B) \overset{\pi_{i*}}{\underset{\pi^*_i}\rightleftarrows} \Dd^b\Coh(G/P_i)$. Then the categorical Demazure operators are defined by $\brmT_i \coloneqq \pi^*_{i}\pi_{i*}$. A standard calculation tells us that $\brmT_{i}$ corresponds to a Fourier-Mukai transform, with its kernel given by $\Tt_i=\Oo_{G/B \times_{G/P_i} G/B} \in \Dd^b\Coh(G/B \times G/B)$, which is the structure sheaf of the Bott-Samelson variety $G/B \times_{G/P_i} G/B$ for all $i$. From \cite[Section 6]{Hsu}, we have the following isomorphisms of Fourier-Mukai kernels
\begin{equation}
\Tt_{i} \cong (\Ee_{i,0} \ast \Ff_{i,1} \ast (\Psi^{+}_{i})^{-1})\bo_{(1,1,...,1)} \cong (\Ff_{i,0} \ast \Ee_{i,-1} \ast (\Psi^{-}_{i})^{-1})\bo_{(1,1,...,1)},  \label{dem} 
\end{equation} where $\Psi_{i}^{+}$, $\Psi_{i}^{-}$ are certain Fourier-Mukai kernels defined in Proposition \ref{Thm catact} for all $1 \leq i \leq N-1$.

Motivating from the idempotent property of Demazure operators, we define ${\TTt}'_{i}\bo_{(1,1,...,1)}:={\E}_{i,0}{\F}_{i,1}({\spi}^{+}_{i})^{-1}\bo_{(1,1,...,1)}$ for an abstract categorical $\dot{\brmU}_{0,N}(L\SL_N)$-action $\Kk$. Then it turns out that they satisfy $({\TTt}'_{i})^2\bo_{(1,1,...,1)} \cong {\TTt}'_{i}\bo_{(1,1,...,1)}$ (see (\ref{eq 9}) for details). This implies that ${\TTt}'_{i}\bo_{(1,1,...,1)}$ is a weak idempotent (see Definition \ref{weakidem}).

To generalize this from categorical $\dot{\brmU}_{0,N}(L\SL_N)$-action to categorical $\dot{\brmU}_{0,N}(L\SL_n)$-action with $n<N$, the key observation is to expressing ${\TTt}'_{i}\bo_{(1,1,...,1)}$ as natural compositions of functors with their left and right adjoints. More precisely, by condition (4) in Definition \ref{catshifted0}, we have 
\begin{equation*}
{\TTt}'_{i}\bo_{(1,1,...,1)} \cong {\E}_{i,0}\bo_{(1,1,...,1)-\alpha_{i}}({\E}_{i,0}\bo_{(1,1,...,1)-\alpha_{i}})^{R}
\end{equation*} then we define ${\TTt}'_{i}\bo_{\kk}$ by the same way, i.e., ${\TTt}'_{i}\bo_{\kk} \coloneqq {\E}_{i,0}\bo_{\kk-\alpha_{i}}({\E}_{i,0}\bo_{\kk-\alpha_{i}})^{R}$, similarly for its complementary ${\SSs}_{i}'\bo_{\kk}$, see (\ref{Tk'}) and (\ref{Sk'}) for details.



\subsection{Application to semiorthogonal decomposition}

Observe that for each $\kk \vDash N$, we have an action of $\Hom(\Kk(\kk),\Kk(\kk))$ on the weight category $\Kk(\kk)$. From Subsection \ref{1.1}, a pair of complementary idempotents gives a short semiorthogonal decomposition; as a consequence of Theorem \ref{mainresult}, we have the following $\Kk(\kk) = \langle \text{Im}{\SSs}'_{i}\bo_{\kk}, \ \text{Im}{\TTt}'_{i}\bo_{\kk} \rangle$.

In particular, when the weight categories are $\Kk(\kk)=\Dd^b\Coh(\Fl_{\kk}(\CC^N))$, the triangulated monoidal category $\Dd^b\Coh(\Fl_{\kk}(\CC^N) \times \Fl_{\kk}(\CC^N))$ acts on $\Dd^b\Coh(\Fl_{\kk}(\CC^N))$ by Fourier-Mukai transforms. Thus we obtain 
\begin{equation} \label{shortsod} 
    \Dd^b\Coh(\Fl_{\kk}(\CC^N)) = \langle \text{Im}\Phi_{{\Ss}'_{i}\bo_{\kk}}, \text{Im}\Phi_{{\Tt}'_{i}\bo_{\kk}} \rangle
\end{equation} for $1 \leq i \leq n-1$. This provides concrete examples where the projection functors of a semiorthogonal decomposition are given by Fourier-Mukai kernels; see Proposition \ref{sodfm}.


It is natural to ask what explicitly are the component categories in (\ref{shortsod}).  We answer this question in the case where $n=2$, i.e., partial flag varieties are the Grassmannians $\Gr(k,N)$. Let $\V$ be the tautological rank $k$ bundle on $\Gr(k,N)$. Then there is a full exceptional collection given by Kapranov \cite{Ka1} 
\begin{equation} \label{kape}
  \Dd^b\Coh(\Gr(k,N))=\langle \s_{\blam}\V \rangle_{\blam \in P(N-k,k)},   
\end{equation} where we denote by $P(N-k,k)$ the set of Young diagrams with at most $k$ rows and at most $N-k$ columns, and $\s_{\blam}$ is the associated Schur functor for each $\blam \in P(N-k,k)$.

Since $n=2$, we drop the subscript $i$ for the generators in $\dot{\brmU}_{0,N}(L\SL_2)$. Then we have the complementary idempotents $(\Tt'\bo_{(k,N-k)},\Ss'\bo_{(k,N-k)})$ in $\Dd^b\Coh(\Gr(k,N) \times \Gr(k,N))$, and we calculate the action of the Fourier-Mukai transform $\Phi_{\Tt'\bo_{(k,N-k)}}$ on the exceptional objects in (\ref{kape}).

\begin{thmx} [Theorem \ref{actofidem}] \label{act}
For $\blam=(\lambda_{1},...,\lambda_{k}) \in P(N-k,k)$, we have 
\begin{equation*}
\Phi_{{\Tt}'\bo_{(k,N-k)}}(\s_{\blam}\V)=\begin{cases}
	0 & \text{if} \ \lambda_{1}=N-k \\
	\s_{\blam}\V & \text{if} \ 0 \leq \lambda_{1} \leq N-k-1.
\end{cases} 
\end{equation*} 
\end{thmx}

This tells us that the Kapranov exceptional collection is a (categorical) eigenbasis for the weak idempotent $\Phi_{{\Tt}'\bo_{(k,N-k)}}$. Since $(\Tt'\bo_{(k,N-k)}, \ \Ss'\bo_{(k,N-k)})$ is a complementary idempotent, we obtain the following corollary.

\begin{cory}[Corollary \ref{componentcat}] \label{com}
The component categories in (\ref{shortsod}) for $\Dd^b\Coh(\Gr(k,N))$ are given by
\begin{equation*}
\emph{Im}\Phi_{{\Tt}'\bo_{(k,N-k)}}= \langle \ \s_{\blam}\V \ \rangle_{0 \leq \lambda_1 \leq N-k-1}, \ 
\emph{Im}\Phi_{{\Ss}'\bo_{(k,N-k)}}= \langle \ \s_{\blam}\V \ \rangle_{\lambda_1 =N-k}.
\end{equation*}
\end{cory}

\subsection{Further remarks and related works}

First, we have to mention that; in fact, we also define other functors ${\TTt}_{i}''\bo_{\kk}$ and ${\SSs}_{i}''\bo_{\kk}$ in $\Hom(\Kk(\kk),\Kk(\kk))$ (see (\ref{Tk''}) and (\ref{Sk''}) for details) by using adjunctions like the definitions of ${\TTt}_{i}'\bo_{\kk}$ and ${\SSs}_{i}'\bo_{\kk}$. Then we prove that all the above results (Theorem \ref{mainresult}, Theorem \ref{act}, and Corollary \ref{com}) for $({\TTt}_{i}''\bo_{\kk}, {\SSs}_{i}''\bo_{\kk})$. Moreover, it turns out that the idempotent triangles given by $({\TTt}_{i}'\bo_{\kk}, \SSs_{i}'\bo_{\kk})$, $({\TTt}_{i}''\bo_{\kk}, {\SSs}_{i}''\bo_{\kk})$ are equivalent to the categorical commutator relations, i.e., conditions (11)(a)(b) in Definition \ref{catshifted0}, see Remark \ref{catcomm'}. 

Furthermore, by study the relations ((A), (B), and (C) in Theorem \ref{mainresult}) for the pairs of complementary idempotents $({\TTt}_{i}'\bo_{\kk},{\SSs}_{i}'\bo_{\kk})$ and $({\TTt}_{i}''\bo_{\kk},{\SSs}_{i}''\bo_{\kk})$, we provide a short proof of the categorical braid relations for $\Tt_i=\Oo_{G/B \times_{G/P_i} G/B}$, see Corollary \ref{shortcatbraid}.

Second, since we only construct categorical $\Uu$-action on the usual derived categories $\Dd^b\Coh(\Fl_{\kk}(\CC^N))$ in \cite{Hsu}. We would like to upgrade this action to the equivariant setting in future work. For example, construct action on the $G$-equivariant derived category $\Dd^b_{G}\Coh(\Fl_{\kk}(\CC^N))$ so that we can extend the results in this article to obtain pairs of complementary idempotents in the triangulated monoidal category 
\begin{equation*}
    \Dd^b\Coh_{G}(\Fl_{\kk}(\CC^N) \times \Fl_{\kk}(\CC^N)) \cong \Dd^b\Coh(P\backslash G / P)
\end{equation*} which is the parabolic coherent Hecke category and $P \subset G$ is the parabolic subgroup such that $G/P \cong \Fl_{\kk}(\CC^N)$. Such categories also appear in works by Arkiphov-Kanstrup \cite{AK1}, \cite{AK2}, and Ben-Zvi-Nalder \cite{BN} with the name quasi-coherent Hecke categories and Demazure Hecke categories, respectively. 

Finally, pairs of complementary idempotents were used by Elias-Hogancamp \cite{EH} for the development of the theory of categorical diagonalization, which serves a primary tool to categorify $q$-Young symmetrizers with application to homological knot invariants, see Abel-Hogancamp \cite{AH}, Elias-Hogancamp \cite{EH1}, and Hogancamp \cite{Ho1}, \cite{Ho2} for related works. It would be interesting to see if there is any application of our pairs of complementary idempotents to homological knot invariant.

\subsection{Acknowledgment}
The author would like to thank Cheng-Chiang Tsai and Chun-Ju Lai for their helpful discussions. The author is supported by MSTC grant 111-2628-M-001-007.

\section{Preliminaries} \label{section 2}

In this section, we provide some background materials for the tools used in the later part of the article.

\subsection{Exceptional collection and semiorthogonal decomposition}

In the first subsection, we give the definitions of semiorthogonal decomposition and exceptional collection. We refer the readers to \cite{Ku1} for a detailed survey.

Let $\Dd$ be a $\CC$-linear triangulated category. 

Given a class $\Ee$ of objects in $\Dd$, we denoted $\langle \Ee \rangle$ to be the minimal full triangulated subcategory of $\Dd$ containing all objects in $\Ee$ and closed under taking direct summands. Similarly for any sequence of full triangulated subcategories $ \Aa_{1},...,\Aa_{n}$ in $\Dd$ we denote by $\langle \Aa_{1},...,\Aa_{n} \rangle$  the minimal full triangulated subcategory of $\Dd$ which contains all of $\Aa_{i},...,\Aa_{n}$.

We begin with the definition of semiorthogonal decompositions.

\begin{definition} 
A  \textit{semiorthogonal decomposition} ($\SOD$ for short) of $\Dd$ is a sequence of full triangulated subcategories $\Aa_{1},...,\Aa_{n}$ such that \begin{enumerate}
    \item there is no non-zero Homs from right to left, i.e.
    $\Hom_{\Dd}(A_{i},A_{j})=0$ for all $A_{i} \in \text{Ob}(\Aa_{i})$, $A_{j} \in \text{Ob}(\Aa_{j})$ where $1 \leq j <  i \leq n$.
    \item $\Dd$ is generated by $\Aa_{1},...,\Aa_{n}$, i.e. the smallest full triangulated subcategory containing $\Aa_{1},...,\Aa_{n}$ is equal to $\Dd$.
\end{enumerate} We will use the notation $\Dd=\langle \Aa_{1},...,\Aa_{n} \rangle$ for a semiorthogonal decomposition of $\Dd$ with components $\Aa_{1},...,\Aa_{n}$.
\end{definition}

A standard way to produce a $\SOD$ is to construct a left or right admissible subcategory, which is the following definition.

\begin{definition}  
A full triangulated subcategory $\Aa \subset \Dd$ is called  \textit{right admissible} if, for the inclusion functor $i:\Aa \rightarrow \Dd$, there is a right adjoint $i^{!}:\Dd \rightarrow \Aa$, and  \textit{left admissible} if there is a left adjoint $i^*:\Dd \rightarrow \Aa$. It is called  \textit{admissible} if it is both left and right admissible.
\end{definition}

We give a remark.

\begin{remark}
A semiorthogonal decomposition $\Dd=\langle \Aa_{1},...,\Aa_{n} \rangle$ is called strong if for each $k$, the category $\Aa_k$ is admissible in $\langle \Aa_{k},...,\Aa_{n} \rangle$. If $X$ is a smooth projective variety, then any semiorthogonal decomposition of the derived category of coherent sheaves $\Dd^b(X)$ is strong.
\end{remark}

Then the following lemma gives us a way to produce $\SOD$s.

\begin{lemma} [\cite{Bon}] \label{lemma 1} 
If $\Dd=\langle \Aa, \Bb \rangle$ is a $\SOD$, then $\Aa$ is left admissible and B is right admissible. Conversely, if $\Aa \subset \Dd$ is left admissible, then $\Dd=\langle \Aa, ^{\perp}\Aa \rangle$ is a $\SOD$, and if $\Bb \subset \Dd$ is right admissible,
then $\Dd=\langle \Bb^{\perp}, \Bb \rangle$ is a $\SOD$.
\end{lemma}

The simplest example of an admissible subcategory is the one generated by an exceptional object.

\begin{definition}
An object $E \in \text{Ob}(\Dd)$ is called  \textit{exceptional} if 
\begin{equation*}
    \Hom_{\Dd}(E,E[l])=\begin{cases}
		\CC & \ \text{if} \ l=0 \\
		0 & \ \text{if} \ l \neq 0.
	\end{cases}
\end{equation*}
\end{definition}

Then we define the notion of exceptional collections.

\begin{definition} 
An ordered collection $\{E_{1},...,E_{n}\}$, where $E_{i} \in \text{Ob}(\Dd)$ for all $1 \leq i \leq n$, is called an \textit{exceptional collection} if each $E_{i}$ is exceptional and $\Hom_{\Dd}(E_{i},E_{j}[l])=0$ for all $i>j$ and $l \in \ZZ$. The collection (sequence) is called  \textit{strong} exceptional if in addition
\begin{equation*}
\Hom_{\Dd}(E_{i},E_{j}[l])=0 \ \text{for all} \ i, \ j \ \text{and} \ l \neq 0.	
\end{equation*}
\end{definition}

Thus, an exceptional collection $\{E_{1},...,E_{n}\}$ in $\Dd$ naturally give rise to a $\SOD$ of $\Dd$
\begin{equation*}
	\Dd=\langle \Aa, E_{1},...,E_{n} \rangle
\end{equation*} where $\Aa=\langle E_{1},...,E_{n} \rangle^{\perp}$ and $E_{i}$ denote the full triangulated subcategory generated by the object $E_{i}$. An exceptional collection is called  \textit{full} if the subcategory $\Aa$ is zero.

Next, we define the notion of (right) dual exceptional collection.
\begin{definition}  \label{definition 7}
	Let $\{E_{1},...,E_{n}\}$ be an exceptional collection on the triangulated category $\Dd$. An exceptional collection $\{F_{1},...,F_{n}\}$ of objects in $\Dd$ is called \textit{right dual} to $\{E_{1},...,E_{n}\}$ if 
	\begin{equation*}
		\Hom_{\Dd}(F_{i},E_{i}[l])=\begin{cases}
			\CC & \ \text{if} \ l=0 \\
			0  & \ \text{if} \ l \neq 0
		\end{cases}
	\end{equation*} and $\Hom_{\Dd}(F_{i},E_{j}[l])=0$ for $i \neq j$ and $l \in \ZZ$.
\end{definition}

\subsection{Fourier-Mukai transforms}

In this subsection, we recall the tool of Fourier-Mukai transforms and kernels. The main reference is the book by Huybrechts \cite{Huy}.

Let $X$ be a smooth and complex projective variety. We will work with the bounded derived category of coherent sheaves on $X$, denoted by $\Dd^b\Coh(X)$. Throughout this article, unless we explicitly mention otherwise, all functors between derived categories will be assumed to be derived. For example, we will write $f^*$ and $\otimes$ instead of $Rf^*$ and $\otimes^{L}$, respectively.

\begin{definition} 
Let $X$ and $Y$ be two smooth and complex projective varieties. A  \textit{Fourier-Mukai kernel} is an object $\Pp$ of the
bounded derived category of coherent sheaves on $X \times Y$. Given $\Pp \in \text{Ob}(\Dd^b\Coh(X \times Y))$, we may define
the associated  \textit{Fourier-Mukai transform}, which is the functor
\begin{align*}
    \Phi_{\Pp}:&\Dd^b\Coh(X) \rightarrow \Dd^b\Coh(Y)  \\
    &\Ff \mapsto \pi_{2*}(\pi_{1}^*(\Ff) \otimes \Pp)
\end{align*} where $\pi_{1}$, $\pi_{2}$ are natural projections from $X \times Y$ to $X$, $Y$ respectively.
\end{definition}

We call $\Phi_{\Pp}$ the Fourier-Mukai transform with (Fourier-Mukai) kernel $\Pp$. For convenience, we will just write FM for Fourier-Mukai.

The first property of FM transforms is that they have left and right adjoints that are themselves FM transforms.

\begin{proposition} \cite[Proposition 5.9]{Huy}  \label{Proposition 1}
For $\Phi_{\Pp}:\Dd^b\Coh(X) \rightarrow \Dd^b\Coh(Y)$ is the FM transform with kernel $\Pp$, define 
\begin{equation*}
	\Pp_{L}=\Pp^{\vee} \otimes \pi^*_{2}\omega_{Y}[\tdim  Y], \ \Pp_{R}=\Pp^{\vee} \otimes \pi^*_{1}\omega_{X}[\tdim X].
\end{equation*} Then
\begin{equation*}
\Phi_{\Pp_{L}}:\Dd^b\Coh(Y) \rightarrow \Dd^b\Coh(X), \  \Phi_{\Pp_{R}}:\Dd^b\Coh(Y) \rightarrow \Dd^b\Coh(X)
\end{equation*} are the left and right adjoints of $\Phi_{\Pp}$, respectively.
\end{proposition}

The second property is that the composition of FM transforms is also a FM transform.

\begin{proposition} \cite[Proposition 5.10]{Huy}  \label{Proposition 2}
Let $X, Y, Z$ be smooth projective varieties over $\CC$. Consider objects $\Pp \in \Rm{Ob}(\Dd^b\Coh(X \times Y))$ and $\Qq \in \Rm{Ob}(\Dd^b\Coh(Y \times Z))$. They define FM transforms $\Phi_{\Pp}:\Dd^b\Coh(X) \rightarrow \Dd^b\Coh(Y)$, $\Phi_{\Qq}:\Dd^b\Coh(Y) \rightarrow \Dd^b\Coh(Z)$.  We use $\ast$ to denote the operation for convolution, i.e.
\begin{equation*}
\Qq \ast \Pp:=\pi_{13*}(\pi_{12}^*(\Pp)\otimes \pi_{23}^{*}(\Qq)).
\end{equation*} Then for $\R=\Qq \ast \Pp \in \Dd^b(X \times Z)$, we have $\Phi_{\Qq} \circ \Phi_{\Pp} \cong \Phi_{\R}$. 
\end{proposition}

\begin{remark} \label{remark 1}
Moreover by \cite[Remark 5.11]{Huy}, we have $(\Qq \ast \Pp)_{L} \cong (\Pp)_{L} \ast (\Qq)_{L}$ and $(\Qq \ast \Pp)_{R} \cong (\Pp)_{R} \ast (\Qq)_{R}$.
\end{remark}

Finally, we state a result by Kuznetsov, which relates the notion of FM transforms and $\SOD$s. Since all the varieties we work with are smooth, we decide to state the following result by Kuznetsov in a weak version. 

\begin{proposition}\cite[Theorem 7.1, Theorem 7.3]{Ku} \label{sodfm}
Let $X$ be a smooth projective variety and $\Dd^b\Coh(X)=\langle \Aa_1,...,\Aa_m \rangle$ be a $\SOD$. Let $p_i:\Dd^b\Coh(X) \rightarrow \Dd^b\Coh(X)$ be the projection to the $i$th component. Then there exists $K_i \in \Rm{Ob}(\Dd^b\Coh(X \times X))$ such that $p_i \cong \Phi_{K_i}$, which is the FM transform with kernel $K_i$.     
\end{proposition}

\subsection{Complementary idempotents}
In this subsection, we recall the work by M. Hogancamp \cite{Ho} about categorical idempotents in a triangulated monoidal category. Here, by ``triangulated monoidal category", we mean a triangulated category with a monoidal structure subject to some compatibility conditions. We refer the reader to Definition 3.1 in \textit{loc. cit.} for details. 

Now we assume $\Dd$ is a $\CC$-linear triangulated category with a monoidal structure $\ast$, and we denote $\bo_{\Dd}$ to be the monoidal identity in $\Dd$, i.e., $\bo_{\Dd} \ast E \cong E \cong E \ast \bo_{\Dd}$.

\begin{definition} \cite[Definition 4.1]{Ho} \label{weakidem}
A \textit{weak idempotent} is an object $E \in \text{Ob}(\Dd)$ such that $E \ast E \cong E$. If $E \in \text{Ob}(\Dd)$ is a weak idempotent, we denote $E\Dd$ to be the full subcategory consisting of objects $A$ such that $E \ast A \cong A$. Similarly, define $\Dd E$ and $E \Dd E$.
\end{definition}

\begin{definition} \cite[Definition 5.4]{BS} 
\cite[Definition 4.2] {Ho} \label{compleidem}
A pair of complementary idempotents is a pair of objects $(\bS, \bT)$ in $\Dd$ such that $\bS \ast \bT \cong 0 \cong \bT \ast \bS$, together with a exact triangle
\begin{equation} \label{idemtri}
    \bT \xrightarrow{\epsilon} \bo_{\Dd} \xrightarrow{\eta} \bS \xrightarrow{\delta} \bT[1]. 
\end{equation} Sometimes we call (\ref{idemtri}) \textit{idempotent triangle}. Also, we call $(\bT,\epsilon)$ a \textit{counital idempotent} and $(\bS,\eta)$ a \textit{unital idempotent}.
\end{definition}

\begin{remark}
The notion of (co)unital idempotents in Definiton \ref{compleidem} also appears in other works with different terminologies. For example, they are called open and closed idempotents in \cite{BD}, and idempotent monads and comonads in \cite{BS}.
\end{remark}

\begin{remark}
It has been shown (Theorem 4.5 in \cite{Ho}) that the exact triangle (\ref{idemtri}) is determined by the isomorphism class $[\bT]$ (or $[\bS]$) in $K(\Dd)$ up to unique isomorphism of triangles. 
\end{remark}

We list the following results about complementary idempotents.

\begin{lemma} \cite[Lemma 2.18]{BD}
Let $E \in \Rm{Ob}(\Dd)$ be a unital or counital idempotent. Then $E \Dd E$ has the structure of a monoidal category with monoidal identity $E$.
\end{lemma}

\begin{proposition}\cite[Theorem 1.2]{Ho} \label{cohTS}
Let $(\bT,\bS)$ be a pair of complementary idempotents in $\Dd$. Then we have 
\begin{align}
\Hom_{\Dd}(\bT,\bT)&\cong \Hom_{\Dd}(\bT,\bo_{\Dd}) \label{coT}\\
\Hom_{\Dd}(\bS,\bS)&\cong \Hom_{\Dd}(\bo_{\Dd},\bS) \label{coS} \\
    \Hom_{\Dd}(\bT,\bS)&\cong 0. \label{orthTS}
\end{align} Moreover, the endomorphism algebras in (\ref{coT}) and (\ref{coS}) are called the cohomology of $\bT$ and $\bS$, respectively.
\end{proposition}

\begin{example} \cite[Example 5.6]{BS} \cite[Remark 4.17]{Ho} \label{ssod} 
Let $\MM$ be a triangulated category on which $\Dd$ acts by triangulated endofunctors, and $(\bT,\bS)$ be a pair of complementary idempotents in $\Dd$. Then we have the $\SOD$: $\MM=\langle \Rm{Im}\bS, \Rm{Im}\bT \rangle$.
\end{example}

\section{Affine 0-Heck algebra and its action} \label{section 3}

In this section, we recall the definition of affine 0-Hecke algebra and its action on the Grothendieck group of full flag variety. 

Fix $N \geq 2$ be a positive integer. We begin with the definition of affine 0-Hecke algebra.

\begin{definition} \label{0-Hecke} 
The affine 0-Hecke algebra $\Hh_{N}(0)$ is defined to be the unital associative $\CC$-algebra generated by the elements $T_{1},...,T_{N-1}$ and the polynomial algebra $\CC[X_{1}^{\pm 1},...,X_{N}^{\pm 1}]$ subject to the following relations
\begin{equation}  \tag{H01}   \label{H01}
T_{i}^2=T_{i},
\end{equation}
\begin{equation}  \tag{H02}   \label{H02}
 T_{i}T_{i+1}T_{i}=T_{i+1}T_{i}T_{i+1}, \ T_{i}T_{j}=T_{j}T_{i}, \ \text{if} \ |i-j| \geq 2, 
\end{equation}
\begin{equation}  \tag{H03}   \label{H03}
	T_{i}X_{j}=X_{j}T_{i} \ \text{if} \ j \neq i,i+1,
\end{equation}
\begin{equation}  \tag{H04}   \label{H04}
	X_{i+1}T_{i}=T_{i}X_{i}+X_{i+1},
\end{equation}
\begin{equation}  \tag{H05}   \label{H05}
	X_{i}T_{i}=T_{i}X_{i+1}-X_{i+1}.
\end{equation}
\end{definition}

\begin{remark}
The affine 0-Hecke algebra $\Hh_N(0)$ is a specialization of the affine Hecke algebra at $q=0$. Moreover, the relations (\ref{H04}) and (\ref{H05}) are usually called the Bernstein-Lusztig relations.
\end{remark}

Next, we explain there is an action of $\Hh_{N}(0)$ on the Grothendieck group of the full flag variety. Let $G=\SLL_{N}(\CC)$, $B \subset G$ be the Borel subgroup of upper triangular matrices, and $T \subset G$ be the maximal torus of diagonal matrices. Then the full flag variety, defined to be quotient $G/B$, can be identified with the following space
\begin{equation} \label{fullflag}
   G/B=\{0 \subset V_{1} \subset V_{2} \subset ... \subset V_{N}=\CC^N \ | \ \dim V_{k}=k \ \text{for} \ \text{all} \ k  \}.
\end{equation}

The Weyl group $W$ is isomorphic to the symmetric group, denoted by $S_N$. Let $\Sigma=\{\alpha_{1},...,\alpha_{N-1}\}$ to be the set of simple roots with the associated simple reflections $\{s_1,...,s_{N-1}\}$ that generate the Weyl group $W=S_{N}$. Then for each $1 \leq i \leq N-1$, there is the minimal parabolic subgroup $P_i=B \cup Bs_iB$. Similarly, the quotient $G/P_i$, which is called a partial flag variety, can be identified with the following space
\begin{equation}  \label{partialflag}
G/P_i=\{0 \subset V_{1} \subset V_{2} \subset ...\subset V_{i-1} \subset V_{i+1} \subset ... \subset  V_{N}=\CC^N \ | \ \dim V_{k} =k \ \text{for} \ k \neq i \}.
\end{equation}

There are natural projections $\pi_{i}:G/B \rightarrow G/P_i$, which is a $\PP^1$-fibration, for $1 \leq i \leq N-1$. 
They induce natural pullback $\pi^{*}_{i}:K(G/P_i) \rightarrow K(G/B)$ and pushforward $\pi_{i*}:K(G/B) \rightarrow K(G/P_i)$ for $1 \leq i \leq N-1$. 

Then we obtain the following operators on $K(G/B)$, which are called the Demazure operates, and the construction is originally due to Bernstein-Gelfand-Gelfand \cite{BGG} in cohomology. 
\begin{equation*}
\delta_{i}=\pi^{*}_{i}\pi_{i*}:K(G/B) \rightarrow K(G/B).   
\end{equation*} 

Furthermore, we have the following result.
\begin{proposition}[\cite{D1}, \cite{D2}]
Those operators satisfy the following
\begin{align*}
&\delta^2_i=\delta_i, \\
&\delta_i\delta_{i+1}\delta_i=\delta_{i+1}\delta_{i}\delta_{i+1}.
\end{align*} More precisely, $\{\delta_i\}$ generate the 0-Hecke algebra and acting on $K(G/B)$.
\end{proposition}

We denote $\V_{i}$ to be the tautological bundle of rank $i$ on $G/B$ for $1 \leq i \leq N$, then we have the natural line bundles $\LL_{i}=\V_{i}/\V_{i-1}$ on $G/B$ for $1 \leq i \leq N$. Let $a_i=[\LL_{i}] \in K(G/B)$ be the class in the Grothendieck group. Then $K(G/B)$ has the following presentation, which is the so-called Borel presentation
\begin{equation} \label{Borpre}
K(G/B) \cong \CC[a_1^{\pm},...,a_N^{\pm}]/ \left \langle e_i- {N \choose i} \right \rangle	
\end{equation} where $\langle e_i- {N\choose i} \rangle$ is the ideal generated by $\{e_i- {N\choose i}\}_{i=1}^{N}$, and $e_i$ is the $i$-th symmetric polynomial in $a_1, a_2,...,a_{N}$.

Under the presentation (\ref{Borpre}), $\delta_{i}$ has the following explicit description
\begin{equation} \label{dema}
\delta_{i}=\frac{a_{i+1}-a_is_{i}}{a_{i+1}-a_i}	
\end{equation} where $s_i$ is the simple reflection that permute $a_i$ and $a_{i+1}$ for all $ 1 \leq i \leq N-1$. 

Finally, we define the operators $x_{j}:K(G/B) \rightarrow K(G/B)$ to be multiplication by $a_{j}$ for all $1 \leq j \leq N$. Then with the help of the expression (\ref{dema}), we can verify the following relations.
\begin{align*}
& a_{i}a_{j}=a_{j}a_{i} \  \text{for all} \ i,j,    
&\delta_{i}a_{j}=a_{j}\delta_{i} \ \text{if} \ j \neq i,i+1,   \\ 
&a_{i+1}\delta_{i}=\delta_{i}a_{i}+a_{i+1}, &a_{i}\delta_{i}=\delta_{i}a_{i+1}-a_{i+1}.
\end{align*}

Thus we conclude that
\begin{corollary} \label{actaffineoHecke}
There is a action of $\Hh_N(0)$ on $K(G/B)$ where the generators $T_i$ and $X_j$ acting by $\delta_i$ and $a_j$, respectively.
\end{corollary}

\section{Shifted 0-affine algebras and its categorical action} \label{section 4}

In this section, we recall the definitions of the shifted 0-affine algebra and its categorical action on the derived category of partial flag varieties. 

Before we go to the detailed setup and the definition of the shifted 0-affine algebra, we should give the readers some background and motivation for it. The main motivation comes from the study of categorical action or categorification of the quantum group $\brmU_{q}(\cg)$ for $\cg$ a semisimple or Kac-Moody Lie algebra.

Based on the work \cite{BLM}, Beilinson-Lusztig-MacPherson give a geometric model for $\brmU_{q}(\SL_{2})$ (more generally $\brmU_{q}(\SL_{n})$). This can be used to construct categorical $\SL_{2}$-action. 

The weight space $V_{\lambda}$ is replaced by the weight category $\Cc(\lambda)=\Dd^bCon(\Gr(k,N))$, which is the bounded derived category of constructible sheaves on the Grassmannian $\Gr(k,N)$ with $\lambda=N-2k$. The generators $e$, $f$ act on them by using the following correspondence diagram

\begin{equation} \label{diag 1}
	\xymatrix{ 
		&&\Fl(k-1,k)=\{0 \overset{k-1}{\subset} V' \overset{1}{\subset} V \overset{N-k}{\subset} \CC^N \} 
		\ar[ld]_{p_1} \ar[rd]^{p_2}   \\
		& \Gr(k,N)  && \Gr(k-1,N)
	}
\end{equation}  where $\Fl(k-1,k)$ is the 3-step partial flag variety and $p_{1}$, $p_{2}$ are the natural projections. Then we define 
${\E}:=p_{2*}p^{*}_{1}$ and a similar functor ${\F}$ in the opposite direction that can be viewed as a lift of $e$ and $f$, respectively.  The functors ${\E}$, ${\F}$  satisfy the defining relations of $\SL_{2}$, i.e., we have 
\begin{equation*} 
	{\E\F}|_{\Cc(\lambda)} \cong {\F\E}|_{\Cc(\lambda)} \bigoplus \text{Id}^{\oplus \lambda}_{\Cc(\lambda)} \ \text{if} \ \lambda \geq 0,
\end{equation*} similarly for $\lambda \leq 0$. 

\begin{remark}
In fact, the above action can be upgraded to categorical $\brmU_{q}(\SL_2)$-action if we add the extra structure of homological shift in the derived category.
\end{remark}

The shifted 0-affine algebra arises naturally when we replace constructible sheaves with coherent sheaves, i.e., we consider $\Dd^b\Coh(\Gr(k,N))$. In the rest of this article, we will omit $\Coh$ and use $\Dd^b(X)$ to denote the bounded derived category of coherent sheaves on a variety $X$.

In this setting, there are more functors. Let $\V, \ \V'$ to be the tautological bundles on $\Fl(k-1,k)$ of rank $k$, $k-1$ respectively. Then there is a natural line bundle $\V/\V'$ on $\Fl(k-1,k)$. Using the same correspondence (\ref{diag 1}), instead of just pulling back and pushing forward directly, we have an extra twist by the line bundles $(\V/\V')^{r}$ where $r \in \ZZ$. So we have the functors
\begin{equation*} 
{\E}_{r}:=p_{2*}(p_{1}^{*}\otimes(\V/\V')^{r}):\Dd^{b}(\Gr(k,N)) \rightarrow \Dd^{b}(\Gr(k-1,N))
\end{equation*} and similarly ${\F}_{r}$ where $r \in \ZZ$.

Then the shifted 0-affine algebra is a $L\SL_{2}:=\SL_{2} \otimes \CC[t,t^{-1}]$-like algebra acting on $\bigoplus_{k} \Dd^{b}(\GG(k,N))$, where $e \otimes t^{r}$ and $f \otimes t^{s}$ acting via the functors ${\E}_{r}$, ${\F}_{s}$ respectively for $r, \ s \in \ZZ$. In \cite{Hsu}, the author is the first one to study this action in detail.  



\subsection{Shifted $q=0$ affine algebras}

In this section, we define the shifted 0-affine algebras. The presentation we use here is by finite numbers of generators and defining relations. In \cite{Hsu}, we also give a conjectural presentation defined by generating series and conjecture that the two presentations are equivalent, see conjecture A.2. in  \textit{loc. cit.}.

Similarly to the dot version $\dot{\brmU}_{q}(\SL_2)$ of $\brmU_{q}(\SL_2)$ that introduced in \cite{BLM}, the shifted 0-affine algebras we introduce below is also an idempotent form. This means that we replace the identity with the direct sum of a system of projectors, one for each element of the weight lattices. They are orthogonal idempotents for approximating the unit element. We refer to part IV in \cite{Lu1} for details of such modification.

Throughout the rest of this article, we fix a positive integer $N \geq 2$. Let
\[
C(n,N):=\{\underline{k}=(k_1,...,k_{n}) \in (\NN \cup \{0\})^n\ | \ k_{1}+...+k_{n}=N \}.
\] 

We regard each $\underline{k}$ as a weight for $\SL_{n}$ via the identification of the weight lattice of $\SL_n$ with the quotient $\ZZ^n/(1,1,...,1)$. We choose the simple root $\alpha_{i}$ to be $(0...0,-1,1,0...0)$ where the $-1$ is in the $i$th position for $1 \leq i \leq n-1$. Finally, we denote $\langle\cdot,\cdot\rangle:\ZZ^n \times \ZZ^n \rightarrow \ZZ$ to be the standard pairing. Then we introduce the shifted 0-affine algebra for $\SL_n$, which is defined by using finite generators and relations.

\begin{definition} \label{shited0affine}
We consider formal symbols of the form $1_{\lambda}x1_{\mu}$ ($\lambda, \mu \in (\NN \cup \{0\})^n$) and abbreviating $(1_{\lambda_1}x_{1}1_{\mu_1})...(1_{\lambda_i}x_{i}1_{\mu_i})=x_{1}...x_{i}1_{\mu_i}$ if the product is nonzero. Then we define the \textit{shifted 0-affine algebra}, denote by $\Uu=\dot{\Rm{\bold{U}}}_{0,N}(L\SL_n)$, to be the associative $\CC$-algebra generated by
\begin{equation*}
\bigcup_{\kk \in C(n,N)}\{1_{\kk}, 1_{\kk+\alpha_{i}}e_{i,r}1_{\kk}, \ 1_{\kk-\alpha_{i}}f_{i,s}1_{\kk},\ 1_{\kk}(\psi^{+}_{i})^{\pm 1}1_{\kk}, \ 1_{\kk}(\psi^{-}_{i})^{\pm 1}1_{\kk},\ 1_{\kk}h_{i,\pm 1}1_{\kk} \}_{1 \leq i \leq n-1}^{-k_{i}-1 \leq r \leq 0, \ 0 \leq s \leq k_{i+1}+1}
\end{equation*} with the following relations
\begin{equation} \tag{U01}   \label{U01}
1_{\kk}1_{\Ll}=\delta_{\kk,\Ll}1_{\kk}, 
\end{equation}
\begin{equation}\tag{U02}   \label{U02}
	\{(\psi^{+}_{i})^{\pm 1}1_{\kk},(\psi^{-}_{i})^{\pm 1}1_{\kk},h_{i,\pm 1}1_{\kk}\ | \ 1\leq i \leq n-1,\kk \in C(n,N)\} \ \mathrm{pairwise\ commute},
\end{equation}
\begin{equation} \tag{U03} \label{U03}
	(\psi^{+}_{i})^{\pm 1} \cdot (\psi^{+}_{i})^{\mp 1} 1_{\kk} = 1_{\kk}=(\psi^{-}_{i})^{\pm 1} \cdot (\psi^{-}_{i})^{\mp 1}1_{\kk},
	\end{equation}
	\begin{align*} \tag{U04} \label{U04}
	&e_{i,r}e_{j,s}1_{\kk}=\begin{cases}
	-e_{i,s+1}e_{i,r-1}1_{\kk} & \text{if} \ j=i \\
	e_{i+1,s}e_{i,r}1_{\kk}-e_{i+1,s-1}e_{i,r+1}1_{\kk} & \text{if} \ j=i+1 \\
	e_{i,r+1}e_{i-1,s-1}1_{\kk}-e_{i-1,s-1}e_{i,r+1}1_{\kk} &\text{if} \ j=i-1 \\
	e_{j,s}e_{i,r}1_{\kk} &\text{if} \ |i-j| \geq 2
	\end{cases},
	\end{align*}
	
	\begin{align*}\tag{U05} \label{U05}
	&f_{i,r}f_{j,s}1_{\kk}=\begin{cases}
	-f_{i,s-1}f_{i,r+1}1_{\kk} & \text{if} \ j=i \\
	f_{i,r-1}f_{i+1,s+1}1_{\kk}-f_{i+1,s+1}f_{i,r-1}1_{\kk} & \text{if} \ j=i+1 \\
	f_{i-1,s}f_{i,r}1_{\kk}-f_{i-1,s+1}f_{i,r-1}1_{\kk} &\text{if} \ j=i-1 \\
	f_{j,s}f_{i,r}1_{\kk} &\text{if} \ |i-j| \geq 2
	\end{cases},
	\end{align*}

	\begin{align*} \tag{U06} \label{U06}
	&\psi^{+}_{i}e_{j,r}1_{\kk}=
	\begin{cases}
	-e_{i,r+1}\psi^{+}_{i}1_{\kk} & \text{if} \ j=i \\
	-e_{i+1,r-1}\psi^{+}_{i}1_{\kk} & \text{if} \ j=i+1 \\
	e_{i-1,r}\psi^{+}_{i}1_{\kk}  &  \text{if} \ j=i-1 \\
	e_{j,r}\psi^{+}_{i}1_{\kk}  &  \text{if} \ |i-j| \geq 2 \\
	\end{cases},
	&\psi^{-}_{i}e_{j,r}1_{\kk}=
	\begin{cases}
	-e_{i,r+1}\psi^{-}_{i}1_{\kk} & \text{if} \ j=i \\
	e_{i+1,r}\psi^{-}_{i}1_{\kk} & \text{if} \ j=i+1 \\
	-e_{i-1,r-1}\psi^{-}_{i}1_{\kk}  &  \text{if} \ j=i-1 \\
	e_{j,r}\psi^{-}_{i}1_{\kk}  &  \text{if} \ |i-j| \geq 2 \\
	\end{cases},
	\end{align*}

	\begin{align*} \tag{U07} \label{U07}
	&\psi^{+}_{i}f_{j,r}1_{\kk}=
	\begin{cases}
	-f_{i,r-1}\psi^{+}_{i}1_{\kk} & \text{if} \ j=i \\
	-f_{i+1,r+1}\psi^{+}_{i}1_{\kk} & \text{if} \ j=i+1 \\
	f_{i-1,r}\psi^{+}_{i}1_{\kk}  &  \text{if} \ j=i-1 \\
	f_{j,r}\psi^{+}_{i}1_{\kk}  &  \text{if} \ |i-j| \geq 2 \\
	\end{cases},
	&\psi^{-}_{i}f_{j,r}1_{\kk}=
	\begin{cases}
	-f_{i,r-1}\psi^{-}_{i}1_{\kk} & \text{if} \ j=i \\
	f_{i+1,r}\psi^{-}_{i}1_{\kk} & \text{if} \ j=i+1 \\
	-f_{i-1,r+1}\psi^{-}_{i}1_{\kk}  &  \text{if} \ j=i-1 \\
	f_{j,r}\psi^{-}_{i}1_{\kk}  &  \text{if} \ |i-j| \geq 2 \\
	\end{cases},
	\end{align*}

	\begin{align*} \tag{U08} \label{U08} 
	&[h_{i,\pm 1},e_{j,r}]1_{\kk}=\begin{cases}
	0 & \text{if} \ i=j \\
	-e_{i+1,r\pm 1}1_{\kk} & \text{if} \ j=i+1 \\
	e_{i-1,r\pm 1}1_{\kk} & \text{if} \  j=i-1 \\
	0& \text{if} \ |i-j| \geq 2 \\
	\end{cases}, 
	&[h_{i,\pm 1},f_{j,r}]1_{\kk}=\begin{cases}
	0 & \text{if} \ i=j \\
	f_{i+1,r\pm 1}1_{\kk} & \text{if} \ j=i+1 \\
	-f_{i-1,r\pm 1}1_{\kk} & \text{if} \  j=i-1 \\
	0& \text{if} \ |i-j| \geq 2 \\
	\end{cases},
	\end{align*}
	
	\begin{equation} \tag{U09} \label{U09} 
	[e_{i,r},f_{j,s}]1_{\kk}=0 \ \mathrm{if} \ i \neq j\ \ \mathrm{and}\ \  [e_{i,r},f_{i,s}]1_{\kk}= \begin{cases}
	\psi^{+}_{i}h_{i,1}1_{\kk} & \text{if} \  r+s=k_{i+1}+1 \\
	\psi^{+}_{i}1_{\kk} & \text{if} \ r+s=k_{i+1} \\
	0 & \text{if} \  -k_{i}+1 \leq r+s \leq k_{i+1}-1 \\
	-\psi^{-}_{i}1_{\kk} & \text{if} \ r+s=-k_{i} \\
	-\psi^{-}_{i}h_{i,-1}1_{\kk} & \text{if} \ r+s=-k_{i}-1
	\end{cases},
	\end{equation}
	
	for any $1 \leq i,j \leq n-1$ and $r,s$ such that the above relations make sense.
\end{definition}

We give an example of the shifted 0-affine algebra.

\begin{example} \label{example 1} 
The shifted 0-affine algebra $\dot{\Rm{\bold{U}}}_{0,2}(L\SL_2)$ is the path algebra of the following quiver
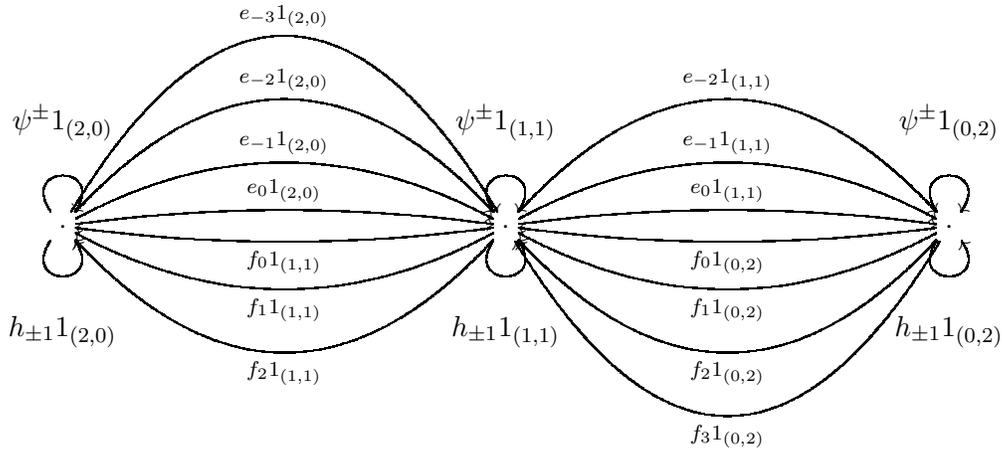
\begin{equation*} 
\xymatrixcolsep{10pc}
\xymatrix{ 
    \psi^{\pm}1_{(2,0)} & \psi^{\pm}1_{(1,1)} & \psi^{\pm}1_{(0,2)} \\
    \cdot  \ar@(ul,ur) \ar@(dl,dr)
    \ar@/^/[r]^{e_{0}1_{(2,0)}} \ar@/^2pc/[r]^{e_{-1}1_{(2,0)}} \ar@/^4pc/[r]^{e_{-2}1_{(2,0)}} \ar@/^6pc/[r]^{e_{-3}1_{(2,0)}}            
	& \cdot \ar@(ul,ur) \ar@(dl,dr) \ar@/^/[r]^{e_{0}1_{(1,1)}}  \ar@/^2pc/[r]^{e_{-1}1_{(1,1)}} \ar@/^4pc/[r]^{e_{-2}1_{(1,1)}}  \ar@/^/[l]^{f_{0}1_{(1,1)}} \ar@/^2pc/[l]^{f_{1}1_{(1,1)}} \ar@/^4pc/[l]^{f_{2}1_{(1,1)}}
	& \cdot \ar@(ul,ur) \ar@(dl,dr) \ar@/^/[l]^{f_{0}1_{(0,2)}} \ar@/^2pc/[l]^{f_{1}1_{(0,2)}} \ar@/^4pc/[l]^{f_{2}1_{(0,2)}}
    \ar@/^6pc/[l]^{f_{3}1_{(0,2)}} \\
    h_{\pm1}1_{(2,0)} & h_{\pm1}1_{(1,1)} & h_{\pm1}1_{(0,2)} \\
}
\end{equation*} subject to the relations from (\ref{U01}) to (\ref{U09}). Fro example, we have $[e_{0},f_{2}]1_{(1,1)}=\psi^{+}h_{1}1_{(1,1)}$.
\end{example}

\subsection{Categorical $\Uu$-action}

In this section, we recall the definition of the categorical action for shifted 0-affine algebra that is defined in \cite{Hsu}. Since we will not use any of the categorical relations that involve the elements $h_{i,\pm 1}1_{\kk}$ in the rest of this article, we do not present the conditions that involve the lift $\Hhh_{i,\pm1}$ here. We refer the readers to Definition 3.1. in \textit{loc. cit.} for a full definition of the categorical action.

\begin{definition}  \label{catshifted0}
A categorical $\Uu=\dot{\Rm{\bold{U}}}_{0,N}(L\SL_n)$ action consists of a target 2-category $\Kk$, which is triangulated, $\CC$-linear and idempotent complete. The objects in $\Kk$ are
\[
\mathrm{Ob}(\Kk)=\{\Kk(\kk)\ |\ \kk \in C(n,N) \}
\] where each $\Kk(\kk)$ is also a triangulated category, and each Hom space $\mathrm{Hom}(\Kk(\kk),\Kk(\Ll))$ is also  triangulated. On those objects $\Kk(\kk)$ we impose the following 1-morphisms:
\begin{equation*}
\bo_{\kk}, \ {\E}_{i,r}\bo_{\kk}=\bo_{\kk+\alpha_{i}}{\E}_{i,r}, \ {\F}_{i,s}\bo_{\kk}=\bo_{\kk-\alpha_{i}}{\F}_{i,s}, \ ({\spi}^{\pm}_{i})^{\pm 1}\bo_{\kk}=\bo_{\kk}({\spi}^{\pm}_{i})^{\pm 1}, 
\end{equation*} where $1 \leq i \leq n-1$, $-k_{i}-1 \leq r \leq 0$, $0 \leq s \leq k_{i+1}+1$. Here $\bo_{\kk}$ is the identity 1-morphism of $\Kk(\kk)$. Those 1-morphisms are subject to the following conditions.
\begin{enumerate}
\item The space of maps between any two 1-morphisms is finite-dimensional.
\item If $\alpha=\alpha_{i}$ or $\alpha=\alpha_{i}+\alpha_{j}$ for some $i,j$ with $\langle\alpha_{i},\alpha_{j}\rangle=-1$, then $\bo_{\kk+r\alpha}=0$ for $r \gg 0$ or  $r \ll 0$.
\item Suppose $i \neq j$. If $\bo_{\kk+\alpha_{i}}$ and $\bo_{\kk+\alpha_{j}}$ are nonzero, then $\bo_{\kk}$ and $\bo_{\kk+\alpha_{i}+\alpha_{j}}$ are also nonzero.
\item The left and right adjoints of $\E_{i,r}$ and $\F_{i,s}$ are given by conjugation of $\spi^{\pm}_{i}$ up to homological shifts. More precisely,
\begin{enumerate}
\item $({\E}_{i,r}\bo_{\kk})^{R} \cong \bo_{\kk}{({\spi}^{+}_{i})^{r+1}}{\F}_{i,k_{i+1}+2}({\spi}^{+}_{i})^{-r-2}[-r-1]$ for all $1 \leq i \leq n-1$,
\item $({\E}_{i,r}\bo_{\kk})^{L} \cong \bo_{\kk}({\spi}^{-}_{i})^{r+k_{i}-1}{\F}_{i,0}({\spi}^{-}_{i})^{-r-k_{i}}[r+k_{i}]$ for all $1 \leq i \leq n-1$,
\item $({\F}_{i,s}\bo_{\kk})^{R} \cong \bo_{\kk}({\spi}^{-}_{i})^{-s+1}{\E}_{i,-k_{i}-2}({\spi}^{-}_{i})^{s-2}[s-1]$ for all $1 \leq i \leq n-1$,
\item $({\F}_{i,s}\bo_{\kk})^{L} \cong \bo_{\kk}({\spi}^{+}_{i})^{-s+k_{i+1}-1}{\E}_{i,0}({\spi}^{+}_{i})^{s-k_{i+1}}[-s+k_{i+1}]$ for all $1 \leq i \leq n-1$.
\end{enumerate}
\item 	
\begin{align*}
&({\spi}_{i}^{\pm})^{\pm 1} ({\spi}_{j}^{\pm})^{\pm 1} \bo_{\kk} \cong ({\spi}_{j}^{\pm})^{\pm 1}({\spi}_{i}^{\pm})^{\pm 1} \bo_{\kk} \ \text{for all} \ i, j, \\
&({\spi}_{i}^{+})^{\pm 1} ({\spi}_{i}^{+})^{\mp 1} \bo_{\kk} \cong ({\spi}_{i}^{-})^{\pm 1}({\spi}_{i}^{-})^{\mp 1} \bo_{\kk} \cong \bo_{\kk} \ \text{for all} \ i.
\end{align*}
\item The relations between ${\E}_{i,r}$, ${\E}_{j,s}$ are given by the following
\begin{enumerate}
	\item  \[
			{\E}_{i,r+1}{\E}_{i,s}\bo_{\kk} \cong \begin{cases}
			{\E}_{i,s+1}{\E}_{i,r}\bo_{\kk}[-1] & \text{if} \ r-s \geq 1 \\
			0 & \text{if} \ r=s \\
			{\E}_{i,s+1}{\E}_{i,r}\bo_{\kk}[1] & \text{if} \ r-s \leq -1. 
			\end{cases}
			\]
			\item ${\E}_{i,r}$, ${\E}_{i+1,s}$  would related by the following exact triangle
			\[
			 {\E}_{i+1,s}{\E}_{i,r+1}\bo_{\kk} \rightarrow {\E}_{i+1,s+1}{\E}_{i,r}\bo_{\kk} \rightarrow {\E}_{i,r}{\E}_{i+1,s+1}\bo_{\kk}. 
			\]
			\item
			${\E}_{i,r}{\E}_{j,s}\bo_{\kk} \cong {\E}_{j,s}{\E}_{i,r}\bo_{\kk}$, if $|i-j| \geq 2$.
		\end{enumerate}
		\item The relations between ${\F}_{i,r}$, ${\F}_{j,s}$ are given by the following
		\begin{enumerate}
			\item \[
			{\F}_{i,r}{\F}_{i,s+1}\bo_{\kk} \cong \begin{cases}
			{\F}_{i,s}{\F}_{i,r+1}\bo_{\kk}[1] & \text{if} \ r-s \geq 1 \\
			0 & \text{if} \ r=s \\
			{\F}_{i,s}{\F}_{i,r+1}\bo_{\kk}[-1] & \text{if} \ r-s \leq -1. 
			\end{cases}  
			\]
			\item ${\F}_{i,r}$, ${\F}_{i+1,s}$ are related by the following exact triangles
			\[
			{\F}_{i,r+1}{\F}_{i+1,s}\bo_{\kk} \rightarrow {\F}_{i,r}{\F}_{i+1,s+1}\bo_{\kk} \rightarrow {\F}_{i+1,s+1}{\F}_{i,r}\bo_{\kk} .
			\] 
			\item ${\F}_{i,r}{\F}_{j,s}\bo_{\kk} \cong {\F}_{j,s}{\F}_{i,r}\bo_{\kk}$, if $|i-j| \geq 2$.
		\end{enumerate}

		\item The relations between ${\E}_{i,r}$, $\Psi^{\pm}_{j}$ are given by the following
		\begin{enumerate}
			\item ${\spi}^{\pm}_{i}{\E}_{i,r}\bo_{\kk} \cong {\E}_{i,r+1}{\spi}^{\pm}_{i}\bo_{\kk}[\mp 1]$.
			\item For $|i-j|=1$, we have the following
			\[
			{\spi}^{\pm}_{i}{\E}_{i\pm1,r}\bo_{\kk} \cong {\E}_{i\pm1,r-1}{\spi}^{\pm}_{i}\bo_{\kk}[\pm 1],
			\]
			\[
			{\spi}^{\pm}_{i}{\E}_{i\mp 1,r}\bo_{\kk} \cong {\E}_{i\mp 1,r} {\spi}^{\pm}_{i}\bo_{\kk}.
			\]
			\item ${\spi}^{\pm}_{i}{\E}_{j,r}\bo_{\kk} \cong {\E}_{j,r} {\spi}^{\pm}_{i}\bo_{\kk}$, if $|i-j| \geq 2$.
		\end{enumerate}
		
		\item The relations between ${\F}_{i,r}$, $\Psi^{\pm}_{j}$ are given by the following
		\begin{enumerate}
		\item ${\spi}^{\pm}_{i}{\F}_{i,r}\bo_{\kk} \cong {\F}_{i,r-1}{\spi}^{\pm}_{i}\bo_{\kk}[\pm 1]$.
			\item For $|i-j|=1$, we have the following
			\[
			{\spi}^{\pm}_{i}{\F}_{i\pm1,s}\bo_{\kk} \cong {\F}_{i\pm1,s+1}{\spi}^{\pm}_{i}\bo_{\kk}[\mp 1],
			\]
			\[
			{\spi}^{\pm}_{i}{\F}_{i\mp1,r}\bo_{\kk} \cong {\F}_{i\mp1,r}{\spi}^{\pm}_{i}\bo_{\kk}.
			\]
			\item ${\spi}^{\pm}_{i,}{\F}_{j,r}\bo_{\kk} \cong {\F}_{j,r}{\spi}^{\pm}_{i}\bo_{\kk}$, if $|i-j| \geq 2$.
		\end{enumerate}

\item If $i \neq j$, then ${\E}_{i,r}{\F}_{j,s}\bo_{\kk} \cong 
{\F}_{j,s}{\E}_{i,r}\bo_{\kk}$.
\item For ${\E}_{i,r}{\F}_{i,s}\bo_{\kk}, {\F}_{i,s}{\E}_{i,r}\bo_{\kk} \in \mathrm{Hom}(\Kk(\kk),\Kk(\kk))$, they are related by  exact triangles, more precisely, 
\begin{enumerate}
\item ${\F}_{i,s}{\E}_{i,r}\bo_{\kk} \rightarrow {\E}_{i,r}{\F}_{i,s}\bo_{\kk} \rightarrow {\spi}^{+}_{i}\bo_{\kk}$, if  
$r+s =k_{i+1}$,
\item ${\E}_{i,r}{\F}_{i,s}\bo_{\kk} \rightarrow {\F}_{i,s}{\E}_{i,r}\bo_{\kk} \rightarrow {\spi}^{-}_{i}\bo_{\kk}$, if  $r+s = -k_{i}$,
\item ${\F}_{i,s}{\E}_{i,r}\bo_{\kk} \cong {\E}_{i,r}{\F}_{i,s}\bo_{\kk}$, if $-k_{i}+1 \leq r+s \leq k_{i+1}-1$.
\end{enumerate}
\end{enumerate} for all $r$, $s$ that make the above conditions make sense, and the isomorphisms between functors appear in every condition are abstractly defined, i.e., we do not specify any 2-morphisms that induce those isomorphisms.
\end{definition}

\subsection{Geometric example}

In this section, we mention the main theorem (Theorem 5.2 in \cite{Hsu}), which says that there is a categorical $\Uu$-action on the derived categories of coherent sheaves on partial flag varieties.

For each $\kk \in C(n,N)$, we define the $n$-step partial flag variety
\begin{equation} \label{npfl}
\Fl_{\kk}(\CC^N)\coloneqq \{V_{\bullet}=(0=V_{0} \subset V_1 \subset ... \subset V_{n}=\CC^N) \ | \ \tdim V_{i}/V_{i-1}=k_{i} \ \text{for} \ \text{all} \ i\}.	
\end{equation} 

We denote $Y(\kk)=\Fl_{\kk}(\CC^N)$. Then  the objects $\Kk(\kk)$ of the triangulated $2$-category $\Kk$ in Definition \ref{catshifted0} will be $\Dd^b(Y(\kk))$. 

On $Y(\kk)$ we denote $\V_{i}$ to be the tautological bundle whose fibre over a point $(0=V_{0} \subset V_1 \subset ... \subset V_{n}=\CC^N)$ is $V_{i}$. To define those $1$-morphisms  ${\E}_{i,r}\bo_{\kk}$, ${\F}_{i,s}\bo_{\kk}$, $({\spi}^{\pm}_{i})^{\pm 1}\bo_{\kk}$, we use the language of FM transforms, that means we will define them by using FM kernels. 

The correspondences (\ref{diag 1}) can be generalized to $\SL_n$ case, which is the following 
\begin{equation*} \label{diag 2}
\xymatrix{ 
		&&W^{1}_{i}(\kk)  
		\ar[ld]_{p_1} \ar[rd]^{p_2}   \\
		& \Fl_{\kk}(\CC^N)  && \Fl_{\kk+\alpha_i}(\CC^N)
	}
\end{equation*} where 
\begin{equation}  \label{eq 30}
	W^{1}_{i}(\kk) \coloneqq \{(V_{\bullet},V_{\bullet}') \in  Y(\kk) \times Y(\kk+\alpha_{i}) \ | \ V_{j}=V'_{j} \ \Rm{for} \ j \neq i,  \ V'_{i} \subset V_{i}\}.
\end{equation} Similarly, we have the natural line bundle $\V_{i}/\V'_{i}$ on $W^{1}_{i}(\kk)$, where $\V_{i}$, $\V'_{i}$ are the tautological bundle on $W^{1}_{i}(\kk)$ of rank $k_{1}+..+k_{i}$, $k_{1}+..+k_{i}-1$, respectively. We also have the transpose correspondence $^{\TTt}W^{1}_{i}(\kk) \subset Y(\kk+\alpha_{i}) \times Y(\kk)$. Let $\iota(\kk):W^{1}_{i}(\kk) \hookrightarrow Y(\kk) \times Y(\kk+\alpha_{i})$, $^{\TTt}\iota(\kk):{^{\TTt}}W^{1}_{i}(\kk) \hookrightarrow Y(\kk+\alpha_{i}) \times Y(\kk)$ be the inclusions and $\Delta(\kk):Y(\kk) \rightarrow Y(\kk) \times Y(\kk)$ be the diagonal map. Then we have the following theorem.

\begin{proposition}\cite[Theorem 5.2]{Hsu} \label{Thm catact}
	Let $\Kk$ be the triangulated 2-categories whose nonzero objects are $\Kk(\kk)=\Dd^b(Y(\kk))$ where $\kk \in C(n,N)$. The 1-morphisms ${\E}_{i,r}\bo_{\kk}$, $\bo_{\kk}{\F}_{i,s}$, $({\spi}^{+}_{i})^{\pm 1}\bo_{\kk}$ and $({\spi}^{-}_{i})^{\pm 1}\bo_{\kk}$ are FM transforms with kernels given by 
\begin{align*}
\Ee_{i,r}\bo_{\kk} &\coloneqq \iota(\kk)_{*} (\V_{i}/\V_{i}')^{r} \in \Dd^b(Y(\kk)\times Y(\kk+\alpha_{i})), \\
\bo_{\kk}\Ff_{i,s}&\coloneqq {^{\TTt}}\iota(\kk)_{*} (\V_{i}/\V'_{i})^{s} \in \Dd^b(Y(\kk+\alpha_{i})\times Y(\kk)), \\	
(\Psi^{+}_{i})^{\pm 1}\bo_{\kk}&\coloneqq \Delta(\kk)_{*}\tdet (\V_{i+1}/\V_{i})^{\pm 1}[\pm(1-k_{i+1})] \in \Dd^b(Y(\kk)\times Y(\kk)), \\
(\Psi^{-}_{i})^{\pm 1}\bo_{\kk}&\coloneqq \Delta(\kk)_{*} \tdet (\V_{i}/\V_{i-1})^{\mp 1}[\pm(1-k_{i})] \in \Dd^b(Y(\kk)\times Y(\kk))
\end{align*} respectively, and the 2-morphisms are maps between kernels.  Then this gives a categorical $\dot{\Uu}_{0,N}(L\SL_n)$ action, i.e. the above data satisfy Definition \ref{catshifted0}.
\end{proposition}

\subsection{Application to affine 0-Hecke algebra} \label{subsectionapplicationtoaffine0Hecke}

In this subsection, we recall the second main result in \cite{Hsu}, which is an application of Proposition \ref{Thm catact} to  construct a categorical action of $\Hh_{N}(0)$ on $\Dd^b(G/B)$.

\begin{proposition}\cite[Theorem 6.6]{Hsu}\label{catactaffine0Hecke} 
There is a categorical action of the affine 0-Hecke algebra $\Hh_{N}(0)$ on $\Dd^b(G/B)$ where we lift the action of $\delta_i$ and $a_j$ in Corollary \ref{actaffineoHecke} to functors $\brmT_i$ and $\brmX_j$ that are FM transforms. More precisely, we define $\Tt_{i}=\Oo_{G/B \times_{G/P_{i}} G/B}$ and $\Xx_{j}=\Delta_{*}(\V_{j}/\V_{j-1})$ to be the kernel for $\brmT_i$ and $\brmX_j$ respectively for $1 \leq i \leq N-1$, $1 \leq j \leq N$. Then we have the following categorical relations
\begin{equation}  \label{6.1}
\Tt_{i} \ast \Tt_{i} \cong \Tt_{i},
\end{equation}
\begin{equation}  \label{6.2}
\Tt_{i}\ast \Tt_{i+1} \ast \Tt_{i} \cong \Tt_{i+1} \ast \Tt_{i} \ast \Tt_{i+1}, \
\Tt_{i}\ast \Tt_{j} \cong \Tt_{j}\ast \Tt_{i} \ \text{if} \ |i-j| \geq 2,
\end{equation}
\begin{equation} \label{6.3}
\Tt_{i}\ast \Xx_{j}\cong \Xx_{j} \ast \Tt_{i} \ \text{if} \ j \neq i,i+1,
\end{equation}
We have the following exact triangles in $\Dd^b(G/B \times G/B)$
\begin{equation}  \label{6.4}
\Tt_{i}\ast \Xx_{i} \rightarrow \Xx_{i+1}\ast \Tt_{i} \rightarrow \Xx_{i+1}, 
\end{equation}
\begin{equation} \label{6.5}
	 \Xx_{i} \ast \Tt_{i} \rightarrow  \Tt_{i} \ast \Xx_{i+1}  \rightarrow \Xx_{i+1}.
\end{equation} 
\end{proposition} 

Instead of giving a proof by direct computation, the key insight we used in \cite{Hsu} to prove it is by writing the functors $\brmT_i$ and $\brmX_j$ in terms of functors in the categorical action of shifted 0-affine algebra $\Uu=\dot{\brmU}_{0,N}(L\SL_N)$.

We use the following observation. By the notation (\ref{npfl}), we have the full flag variety $G/B$ is the same as $\Fl_{(1,1,...,1)}(\CC^N)$. Similarly, for the partial flag varieties $G/P_{i}$ given in (\ref{partialflag}), we have $G/P_{i}=\Fl_{(1,1,...,1) + \alpha_{i}}(\CC^N)=\Fl_{(1,1,...,1) - \alpha_{i}}(\CC^N)$. Here $\alpha_{i}$ is the simple root $(0...0,-1,1,0...0)$ where the $-1$ is in the $i$th position for $1 \leq i \leq N-1$. Thus we obtain the following diagram
\begin{equation*}
\xymatrix@C=4em{
	\Dd^b(\Fl_{{(1,1,...,1)-\alpha_{i}}}(\CC^N))   \ar@/^/[r]^{{\E}_{i,r}1_{(1,1,...,1)-\alpha_{i}}}   
	& \Dd^b(\Fl_{{(1,1,...,1)}}(\CC^N))   \ar@/^/[r]^{{\E}_{i,r}1_{(1,1,...,1)}}   \ar@/^/[l]^{{\F}_{i,s}1_{(1,1,...,1)}}
	& \Dd^b(\Fl_{{(1,1,...,1)+\alpha_{i}}}(\CC^N)).   \ar@/^/[l]^{{\F}_{i,s}1_{(1,1,...,1)+\alpha_{i}}}  
	}
\end{equation*} 

Then from Section 6 of \cite{Hsu}, we have the following isomorphisms of FM kernels
\begin{align}
\Tt_{i} &\cong (\Ee_{i,0} \ast \Ff_{i,0})\bo_{(1,1,...,1)} \cong (\Ff_{i,0} \ast \Ee_{i,0})\bo_{(1,1,...,1)} \label{eq p1} \\
&\cong (\Ee_{i,0} \ast \Ff_{i,1} \ast (\Psi^{+}_{i})^{-1})\bo_{(1,1,...,1)} \cong (\Ff_{i,0} \ast \Ee_{i,-1} \ast (\Psi^{-}_{i})^{-1})\bo_{(1,1,...,1)}\label{eq p2} 
\end{align} for all $1 \leq i \leq N-1$. Similarly, $\Xx_{i} \cong \Psi^{+}_{i-1}\bo_{(1,1,...,1)} \cong (\Psi^{-}_{i})^{-1}\bo_{(1,1,...,1)}$ for all $1 \leq i \leq N$.

From conditions (11)(a) and (11)(b), we have the following two exact triangles
\begin{align}
&(\Ff_{i,1}\ast \Ee_{i,0}) \bo_{(1,1,...,1)} \rightarrow (\Ee_{i,0} \ast \Ff_{i,1})\bo_{(1,1,....,1)} \rightarrow \Psi_{i}^{+}\bo_{(1,1,...,1)}, \label{commu1}\\
& (\Ee_{i,-1}\ast \Ff_{i,0}) \bo_{(1,1,...,1)} \rightarrow (\Ff_{i,0} \ast \Ee_{i,-1})\bo_{(1,1,....,1)} \rightarrow \Psi_{i}^{-}\bo_{(1,1,...,1)}. \label{commu2}
\end{align}

Using (\ref{eq p1}) and (\ref{eq p2}), it is standard to check that 
\begin{align*}
&(\Ee_{i,0} \ast \Ff_{i,1})\bo_{(1,1,....,1)} \cong(\Ee_{i,0} \ast \Ff_{i,0} \ast \Psi_{i}^{+})\bo_{(1,1,....,1)} \cong \Tt_{i} \ast \Xx_{i+1},   \\
&(\Ff_{i,1}\ast \Ee_{i,0}) \bo_{(1,1,...,1)} \cong ((\Psi^{-}_{i})^{-1} \ast \Ff_{i,0}\ast \Ee_{i,0}) \bo_{(1,1,...,1)} \cong   \Xx_{i}  \ast \Tt_{i}. 
\end{align*}  Thus (\ref{commu1}) is equivalent to the categorical Bernstein-Lusztig relation (\ref{6.5}). A similar check tells us that (\ref{commu2}) is equivalent to relation (\ref{6.4}).

Hence, the above calculations tell us that for the categorical $\Uu=\dot{\brmU}_{0,N}(L\SL_N)$-action on $\Dd^b(G/B) \bigoplus \bigoplus_{i=1}^{N-1}\Dd^b(G/P_{i})$, the categorical commutator relations in condition (11)(a) and (11)(b) from Definition \ref{catshifted0} are equivalent to the categorified Bernstein-Lusztig relations (\ref{6.4}) and (\ref{6.5}).

For the rest relations, most of them from (\ref{6.1}) to (\ref{6.3}) are standard to verify. The difficult one is the braid relation in (\ref{6.2}). In \cite{Hsu}, we prove the braid relation by using relations in the categorical action of $\dot{\Uu}_{0,N}(L\SL_{N})$, which still involves many calculations. In the next section, we will use the main result to provide a second proof of the braid relation, which drastically reduces the calculations.

\section{Main result: Categorical idempotents} \label{section 5}

In this section, we prove the main result of this article, i.e., a categorical $\Uu$-action gives two families of pair of complementary idempotents for each weight category. The main idea is to generalize the interpretation of the categorified Demazure operators $\brmT_i$ in terms of functors in the categorical $\dot{\brmU}_{0,N}(L\SL_N)$-action in Subsction \ref{subsectionapplicationtoaffine0Hecke} to general $\dot{\brmU}_{0,N}(L\SL_n)$ case.

\subsection{Some observations}

The study of the full flag variety $G/B$ in Subsection \ref{subsectionapplicationtoaffine0Hecke} is our main motivation. Recall the categorified Bernstein-Lusztig relations (\ref{6.4}) and (\ref{6.5}) 
\begin{align*}
 &\Tt_{i}\ast \Xx_{i} \rightarrow \Xx_{i+1}\ast \Tt_{i} \rightarrow \Xx_{i+1}, \\ 
&\Xx_{i} \ast \Tt_{i} \rightarrow  \Tt_{i} \ast \Xx_{i+1}  \rightarrow \Xx_{i+1}.   
\end{align*} Since $\Xx_{i+1}$ are invertible, convolving with $\Xx^{-1}_{i+1}$ and shift them by $[1]$, we get the following exact triangles
\begin{align}
&\Tt_{i} \rightarrow \Oo_{\Delta} \rightarrow \Xx^{-1}_{i+1} \ast \Tt_{i} \ast \Xx_{i}[1], \label{idem1} \\ 
&\Tt_{i}   \rightarrow \Oo_{\Delta} \rightarrow \Xx_{i} \ast \Tt_{i} \ast \Xx^{-1}_{i+1}[1], \label{idem2}   
\end{align} where $\Oo_{\Delta}$ is the structure sheaf of the diagonal in $G/B \times G/B$.

From relation (\ref{6.1}), we know that $\Tt_i$ is a weak idempotent (see Definition \ref{weakidem}) in the triangulated monoidal category $\Dd^b(G/B \times G/B)$, where the monoidal structure is given by the convolution of FM kernels and the monoidal identity is $\Oo_{\Delta}$. Thus (\ref{idem1}) and (\ref{idem2}) are two idempotent triangles (see Definition \ref{compleidem}), and
$(\Tt_i,\Oo_{\Delta})$ is a counital idempotent for all $1 \leq i \leq N-1$. The two exact triangles (\ref{idem1}) and (\ref{idem2}) also imply that $\Xx^{-1}_{i+1} \ast \Tt_{i}\ast \Xx_{i} \cong \Xx_{i} \ast \Tt_{i} \ast \Xx^{-1}_{i+1}$.

On the other hand, since the two exact triangles (\ref{commu1}) and (\ref{commu2}) are equivalent to the categorified Bernstein-Lusztig relations (\ref{6.4}) and (\ref{6.5}), in terms of the categorical $\dot{\brmU}_{0,N}(L\SL_N)$-action (\ref{idem1}) and (\ref{idem2}) are equivalent to the following 
\begin{align*}
&(\Ee_{i,0} \ast \Ff_{i,1} \ast (\Psi_{i}^{+})^{-1})\bo_{(1,1,....,1)} \rightarrow \Oo_{\Delta} \rightarrow (\Ff_{i,1}\ast \Ee_{i,0} \ast (\Psi_{i}^{+})^{-1})\bo_{(1,1,...,1)}[1] \\
& (\Ff_{i,0} \ast \Ee_{i,-1} \ast (\Psi_{i}^{-})^{-1})\bo_{(1,1,....,1)} \rightarrow \Oo_{\Delta} \rightarrow (\Ee_{i,-1}\ast \Ff_{i,0}\ast (\Psi_{i}^{-})^{-1}) \bo_{(1,1,...,1)}[1].   
\end{align*} 

This suggests us that the following two pairs
\begin{align*}
&({\E}_{i,0}{\F}_{i,1}({\spi}_{i}^{+})^{-1}\bo_{(1,1,...,1)}, {\F}_{i,1}{\E}_{i,0}({\spi}_{i}^{+})^{-1}\bo_{(1,1,...,1)}[1]) \\
&({\F}_{i,0}{\E}_{i,-1}({\spi}_{i}^{-})^{-1}\bo_{(1,1,...,1)}, {\E}_{i,-1}{\F}_{i,0}({\spi}_{i}^{-})^{-1}\bo_{(1,1,...,1)}[1])
\end{align*} are complementary idempotents in an abstract categorical $\dot{\brmU}_{0,N}(L\SL_N)$-action $\Kk$. Indeed, if we define 
\begin{equation} \label{T'1}
{\TTt}'_{i}\bo_{(1,1,...,1)}:={\E}_{i,0}{\F}_{i,1}({\spi}^{+}_{i})^{-1}\bo_{(1,1,...,1)}.
\end{equation} Then we calculate 
\begin{align} \label{eq 9}
\begin{split}       ({\TTt}'_{i})^{2}\bo_{(1,1,...,1)}&={\E}_{i,0}{\F}_{i,1}({\spi}^{+}_{i})^{-1}{\E}_{i,0}{\F}_{i,1}({\spi}^{+}_{i})^{-1}\bo_{(1,1,...,1)} \\
	&\cong {\E}_{i,0}{\F}_{i,1}{\E}_{i,-1}({\spi}^{+}_{i})^{-1}{\F}_{i,1}({\spi}^{+}_{i})^{-1}\bo_{(1,1,...,1)}[1] \\
	& \cong {\E}_{i,0}({\spi}_{i})({\spi}^{+}_{i})^{-1}{\F}_{i,1}({\spi}^{+}_{i})^{-1}\bo_{(1,1,...,1)}[1-1] \\
	& \cong {\E}_{i,0}{\F}_{i,1}({\spi}^{+}_{i})^{-1}\bo_{(1,1,...,1)} ={\TTt}'_{i}\bo_{(1,1,...,1)}
	\end{split}
\end{align} where the first isomorphism comes from condition (8)(a) in Definition \ref{catshifted0}. On the other hand, we define 
\begin{equation} \label{T''1}
	{\TTt}''_{i}\bo_{(1,1,...,1)} \coloneqq {\F}_{i,0}{\E}_{i,-1}({\spi}^{-}_{i})^{-1}\bo_{(1,1,...,1)},
\end{equation} and a similar calculation tells us that it satisfies $({\TTt}''_{i})^2\bo_{(1,1,...,1)} \cong {\TTt}''_{i}\bo_{(1,1,...,1)}$. 

Hence both ${\TTt}'_{i}\bo_{(1,1,...,1)}$ and ${\TTt}''_{i}\bo_{(1,1,...,1)}$ are weak idempotents, and we have the following idempotent triangles in $\Hom(\Kk(1,1,...,1),\Kk(1,1,...,1))$.
\begin{align}
&{\TTt}'_{i}\bo_{(1,1,...,1)}={\E}_{i,0}{\F}_{i,1}({\spi}_{i}^{+})^{-1}\bo_{(1,1,....,1)} \rightarrow \bo_{(1,1,...,1)} \rightarrow {\F}_{i,1}{\E}_{i,0}({\spi}_{i}^{+})^{-1} \bo_{(1,1,...,1)}[1] , \label{eq 10}\\
& {\TTt}''_{i}\bo_{(1,1,...,1)}={\F}_{i,0}{\E}_{i,-1}({\spi}_{i}^{-})^{-1} \bo_{(1,1,....,1)} \rightarrow \bo_{(1,1,...,1)} \rightarrow {\E}_{i,-1}{\F}_{i,0}({\spi}_{i}^{-})^{-1} \bo_{(1,1,...,1)}[1]. \label{eq 11}
\end{align}

\begin{remark}
Note that even though we have the isomorphism (\ref{eq p2}), this does not implies that  ${\E}_{i,0}{\F}_{i,1}({\spi}^{+}_{i})^{-1}\bo_{(1,1,...,1)} \cong {\F}_{i,0}{\E}_{i,-1}({\spi}^{-}_{i})^{-1}\bo_{(1,1,...,1)}$ since the categorical action now is on an abstract category $\Kk$. It holds only on the geometric example $\Kk((1,1,...,1))=\Dd^b(G/B)$.
\end{remark}

For convenience, we define the extra idempotents in (\ref{eq 10}) and (\ref{eq 11}) to be  
\begin{align}
{\SSs}'_{i}\bo_{(1,1,...,1)}& \coloneqq {\F}_{i,1}{\E}_{i,0}({\spi}_{i}^{+})^{-1} \bo_{(1,1,...,1)}[1], \\ {
\SSs}''_{i}\bo_{(1,1,...,1)}& \coloneqq {\E}_{i,-1}{\F}_{i,0}({\spi}_{i}^{-})^{-1} \bo_{(1,1,...,1)}[1],
\end{align} then the idempotent property of ${\TTt}'_{i}\bo_{(1,1,...,1)}$ and ${\TTt}''_{i}\bo_{(1,1,...,1)}$ implies that both ${\SSs}'_{i}\bo_{(1,1,...,1)}$ and ${\SSs}''_{i}\bo_{(1,1,...,1)}$ are also weak idempotents. 

We summarize the above discussion in the following corollary.  
\begin{corollary} \label{compidemcat}
Given a categorical $\dot{\brmU}_{0,N}(L\SL_N)$-action $\Kk$. There are two families of complementary idempotents in the functor category $\Hom(\Kk(1,1,...,1),\Kk(1,1,...,1))$, which is a triangulated category by Definition \ref{catshifted0}, and has a monoidal structure given by the composition of functors. They are
\begin{align*}
  &({\TTt}'_{i}\bo_{(1,1,...,1)},{\SSs}'_{i}\bo_{(1,1,...,1)})=({\E}_{i,0}{\F}_{i,1}({\spi}_{i}^{+})^{-1}\bo_{(1,1,....,1)},{\F}_{i,1}{\E}_{i,0}({\spi}_{i}^{+})^{-1} \bo_{(1,1,...,1)}[1]),  \\
  &({\TTt}''_{i}\bo_{(1,1,...,1)},{\SSs}''_{i}\bo_{(1,1,...,1)})=({\F}_{i,0}{\E}_{i,-1}({\spi}_{i}^{-})^{-1} \bo_{(1,1,....,1)},{\E}_{i,-1}{\F}_{i,0}({\spi}_{i}^{-})^{-1} \bo_{(1,1,...,1)}[1])
\end{align*}  for $1 \leq i \leq N-1$. 
\end{corollary}

The goal of this article is to generalize the above observation to a more general categorical $\dot{\brmU}_{0,N}(L\SL_n)$-action setting.

\subsection{The main result}

In this subsection, we generalize Corollary \ref{compidemcat} to general categorical $\dot{\brmU}_{0,N}(L\SL_n)$-action $\Kk$. More precisely, for each weight category $\Kk(\kk)$, we have to construct certain counital idempotents ${\TTt}'_{i}\bo_{\kk}$ and ${\TTt}''_{i}\bo_{\kk}$. 

With the definition of ${\TTt}'_{i}\bo_{(1,1,...,1)}$ and ${\TTt}''_{i}\bo_{(1,1,...,1)}$ from (\ref{T'1}) and (\ref{T''1}), it is not clear to see a direct generalization to weak idempotents ${\TTt}'_{i}\bo_{\kk}$ and ${\TTt}''_{i}\bo_{\kk}$.

We have to use more relations in the categorical $\Uu$-action. Note that by condition (4) in Definition \ref{catshifted0}, ${\E}_{i,r}$ and ${\F}_{i,s}$ both have left and right adjoints. Using conditions (8)(a) and (9)(a) in Definition \ref{catshifted0} and a simple check tells us that 
\begin{align*}
{\E}_{i,0}\bo_{(1,1,...,1)-\alpha_{i}}({\E}_{i,0}\bo_{(1,1,...,1)-\alpha_{i}})^{R} &\cong {\E}_{i,0}{\F}_{i,1}({\spi}^{+}_{i})^{-1}\bo_{(1,1,...,1)}, \\
{\F}_{i,0}\bo_{(1,1,...,1)+\alpha_{i}} ({\F}_{i,0}\bo_{(1,1,...,1)+\alpha_{i}})^{R} &\cong {\F}_{i,0}{\E}_{i,-1}({\spi}^{-}_{i})^{-1}\bo_{(1,1,...,1)}.
\end{align*} Thus we observe that  
\begin{align}
{\TTt}'_{i}\bo_{(1,1,...,1)} \cong {\E}_{i,0}\bo_{(1,1,...,1)-\alpha_{i}}({\E}_{i,0}\bo_{(1,1,...,1)-\alpha_{i}})^{R}, \label{n1} \\ 
{\TTt}''_{i}\bo_{(1,1,...,1)} \cong {\F}_{i,0}\bo_{(1,1,...,1)+\alpha_{i}}({\F}_{i,0}\bo_{(1,1,...,1)+\alpha_{i}})^{R}. \label{n2}
\end{align} 

On the other hand, we also have 
\begin{equation*}
({\F}_{i,0}\bo_{(1,1,...,1)})^{L} \cong {\E}_{i,0}({\spi}^{+}_{i})^{-1}\bo_{(1,1,...,1)-\alpha_i}[1], \ 
({\E}_{i,0}\bo_{(1,1,...,1)})^{L}  \cong {\F}_{i,0}({\spi}^{-}_{i})^{-1}\bo_{(1,1,...,1)+\alpha_i}[1], 
\end{equation*} which implies that 
\begin{equation*}
(({\spi}^{+}_{i})^{-1}{\F}_{i,0}\bo_{(1,1,...,1)}[1])^{L} \cong {\E}_{i,0}\bo_{(1,1,...,1)-\alpha_{i}}, \ (({\spi}^{-}_{i})^{-1}{\E}_{i,0}\bo_{(1,1,...,1)}[1])^{L} \cong {\F}_{i,0}\bo_{(1,1,...,1)+\alpha_{i}} 
\end{equation*}

Thus another way to write (\ref{n1}) and (\ref{n2}) is
\begin{align}
{\TTt}'_{i}\bo_{(1,1,...,1)} \cong (({\spi}^{+}_{i})^{-1}{\F}_{i,0}\bo_{(1,1,...,1)})^{L}({\spi}^{+}_{i})^{-1}{\F}_{i,0}\bo_{(1,1,...,1)}, \label{n1'} \\ 
{\TTt}''_{i}\bo_{(1,1,...,1)} \cong (({\spi}^{-}_{i})^{-1}{\E}_{i,0}\bo_{(1,1,...,1)})^{L}({\spi}^{-}_{i})^{-1}{\E}_{i,0}\bo_{(1,1,...,1)}. \label{n2'} 
\end{align} 

Similarly, the definition of ${\SSs}'_{i}\bo_{(1,1,...,1)}$ and ${\SSs}''_{i}\bo_{(1,1,...,1)}$ also comes from such construction, i.e.
\begin{align} 
{\SSs}'_{i}\bo_{(1,1,...,1)}& \cong ({\E}_{i,0}({\spi}_{i}^{+})^{-1}\bo_{(1,1,...,1)})^{R}{\E}_{i,0}({\spi}_{i}^{+})^{-1}\bo_{(1,1,...,1)}, \\ 
& \cong {\F}_{i,1}\bo_{(1,1,...,1)+\alpha_i}({\F}_{i,1}\bo_{(1,1,...,1)+\alpha_i})^{L}, \label{m1} \\
{\SSs}''_{i}\bo_{(1,1,...,1)}& \cong ({\F}_{i,0}({\spi}^{-}_{i})^{-1}\bo_{(1,1,...,1)})^{R}{\F}_{i,0}({\spi}^{-}_{i})^{-1}\bo_{(1,1,...,1)}, \\
& \cong {\E}_{i,-1}\bo_{(1,1,...,1)-\alpha_i}({\E}_{i,-1}\bo_{(1,1,...,1)-\alpha_i})^{L}. \label{m2}
\end{align}  

Hence, the above calculation tells us that the weak idempotents ${\TTt}'_i\bo_{(1,1,...,1)}$, ${\TTt}''_i\bo_{(1,1,...,1)}$, ${\SSs}'_i\bo_{(1,1,...,1)}$, and ${\SSs}''_i\bo_{(1,1,...,1)}$ can be constructed by composition of functors with their left/right adjoints.

Now, we move to the general case (in the geometric example, from the full flag variety to the $n$-step partial flag varieties); i.e., we have an abstract categorical $\Uu=\dot{\brmU}_{0,N}(L\SL_{n})$ action on $\Kk$ with $ n < N$. Like the study for the full flag variety case, we have the following picture.
\begin{equation*}
\xymatrix@C=5em{ 
	...
    & \Kk(\kk-\alpha_{i})  \ar@/^/[r]^{{\E}_{i,r}\bo_{\kk-\alpha_i}}
    &\Kk(\kk)   \ar@/^/[r]^{{\E}_{i,r}\bo_{\kk}}  \ar@/^/[l]^{{\F}_{i,s}\bo_{\kk}}
    & \Kk(\kk+\alpha_{i})  \ar@/^/[l]^{{\F}_{i,s}\bo_{\kk+\alpha_i}}
	&....   
}
\end{equation*}

To construct the weak idempotents ${\TTt}'_{i}\bo_{\kk}$ and ${\TTt}''_{i}\bo_{\kk}$ in the triangulated monoidal category $\Hom(\Kk(\kk),\Kk(\kk))$, since ${\E}_{i,r}$ and ${\F}_{i,s}$ both admit left/right adjoints, we imitate the natural construction of the weak idempotents ${\TTt}'_{i}\bo_{(1,1,...,1)}$ and ${\TTt}''_{i}\bo_{(1,1,...,1}$ in $\Hom(\Kk(1,1,...,1),\Kk(1,1,...,1))$ above. In particular, we choose (\ref{n1}) and (\ref{n2}) to define our ${\TTt}'_{i}\bo_{\kk}$ and ${\TTt}''_{i}\bo_{\kk}$ since the definition is simple.

Thus we define the following 
\begin{align} 
{\TTt}'_{i}\bo_{\kk} &\coloneqq {\E}_{i,0}{\F}_{i,k_{i+1}}({\spi}^{+}_{i})^{-1}\bo_{\kk} \cong {\E}_{i,0}\bo_{\kk-\alpha_{i}}({\E}_{i,0}\bo_{\kk-\alpha_{i}})^{R}, \label{Tk'} \\
{\TTt}''_{i}\bo_{\kk} &\coloneqq {\F}_{i,0}{\E}_{i,-k_{i}}({\spi}^{-}_{i})^{-1}\bo_{\kk} \cong 
{\F}_{i,0}\bo_{\kk+\alpha_{i}} ({\F}_{i,0}\bo_{\kk+\alpha_{i}})^{R}.\label{Tk''} 
\end{align}

Similarly, we define ${\SSs}'_{i}\bo_{\kk}$ and ${\SSs}''_{i}\bo_{\kk}$ by (\ref{m1}) and (\ref{m2}).
\begin{align}
&{\SSs}'_{i}\bo_{\kk} \coloneqq {\F}_{i,k_{i+1}}{\E}_{i,0}({\spi}^{+}_{i})^{-1}\bo_{\kk}[1] \cong {\F}_{i,k_{i+1}}\bo_{\kk+\alpha_i}({\F}_{i,k_{i+1}}\bo_{\kk+\alpha_i})^{L}, \label{Sk'} \\
&{\SSs}''_{i}\bo_{\kk} \coloneqq {\E}_{i,-k_{i}}{\F}_{i,0}({\spi}^{-}_{i})^{-1}\bo_{\kk}[1] \cong {\E}_{i,-k_{i}}\bo_{\kk-\alpha_i}({\E}_{i,-k_{i}}\bo_{\kk-\alpha_i})^{L}. \label{Sk''}
\end{align} 

The main result of this article is to show that the functors ${\TTt}'_{i}\bo_{\kk}$, ${\TTt}''_{i}\bo_{\kk}$, ${\SSs}'_{i}\bo_{\kk}$, and ${\SSs}''_{i}\bo_{\kk}$ defined above give pairs of complementary idempotents in $\Hom(\Kk(\kk),\Kk(\kk))$. Furthermore, we will study their relations when acting on the abstract category $\Kk(\kk)$.

Note that from conditions (11)(a)(b) in Definition \ref{catshifted0} we have the following two exact triangles like (\ref{eq 10}) and (\ref{eq 11}) 
\begin{align}
& {\TTt}'_{i}\bo_{\kk}={\E}_{i,0}{\F}_{i,k_{i+1}}({\spi}_{i}^{+})^{-1}\bo_{\kk} \xrightarrow{\epsilon} \bo_{\kk} \xrightarrow{\eta} {\SSs}'_{i}\bo_{\kk}={\F}_{i,k_{i+1}}{\E}_{i,0}({\spi}_{i}^{+})^{-1} \bo_{\kk}[1], \label{triTk'} \\
&  {\TTt}''_{i}\bo_{\kk}={\F}_{i,0}{\E}_{i,-k_{i}}({\spi}_{i}^{-})^{-1} \bo_{\kk} \xrightarrow{\epsilon'} \bo_{\kk} \xrightarrow{\eta'} {\SSs}''_{i}\bo_{\kk}={\E}_{i,-k_{i}}{\F}_{i,0}({\spi}_{i}^{-})^{-1} \bo_{\kk}[1]. \label{TriTk''}
\end{align} However, at this moment we are not sure whether $({\TTt}'_{i},{\SSs}'_{i})$ and $({\TTt}''_{i},{\SSs}''_{i})$ both give pair of complementary idempotents.

\begin{remark}
Note that the constructions of the  functors ${\TTt}'\bo_{\kk}$, ${\SSs}'\bo_{\kk}$, ${\TTt}''\bo_{\kk}$, ${\SSs}''\bo_{\kk}$ are similar to the construction of spherical twist/co-twist in \cite{AL}. However, they are not spherical twist/co-twist functors since we will prove that all of them are weak idempotents in the following theorem. 
\end{remark}

The following is the main result of this article.



\begin{theorem} \label{Theorem 5}
Given a categorical $\dot{\brmU}_{0,N}(L\SL_n)$ action on $\Kk$. Then we have the following two families of pairs of   
complementary idempotents in the triangulated monoidal category $\Hom(\Kk(\kk),\Kk(\kk))$ 
\begin{align*}
&({\TTt}'_{i}\bo_{\kk},{\SSs}'_{i}\bo_{\kk})=({\E}_{i,0}\bo_{\kk-\alpha_{i}}({\E}_{i,0}\bo_{\kk-\alpha_{i}})^{R}, {\F}_{i,k_{i+1}}\bo_{\kk+\alpha_i}({\F}_{i,k_{i+1}}\bo_{\kk+\alpha_i})^{L}) \\
&({\TTt}''_{i}\bo_{\kk},{\SSs}''_{i}\bo_{\kk})=({\F}_{i,0}\bo_{\kk+\alpha_{i}} ({\F}_{i,0}\bo_{\kk+\alpha_{i}})^{R},{\E}_{i,-k_{i}}\bo_{\kk-\alpha_i}({\E}_{i,-k_{i}}\bo_{\kk-\alpha_i})^{L})  
\end{align*} for all $1 \leq i \leq n-1$. Moreover, they satisfy the following relations.
\begin{enumerate}[label=(\Alph*)]
        \item For $|i-j| \geq 2$, we have 
	\begin{equation*}
		{\TTt}'_{i}{\TTt}'_{j}\bo_{\kk} \cong {\TTt}'_{j}{\TTt}'_{i}\bo_{\kk}, \ {\SSs}'_{i}{\SSs}'_{j}\bo_{\kk} \cong {\SSs}'_{j}{\SSs}'_{i}\bo_{\kk}, \
		{\TTt}''_{i}{\TTt}''_{j}\bo_{\kk} \cong {\TTt}''_{j}{\TTt}''_{i}\bo_{\kk}, \ {\SSs}''_{i}{\SSs}''_{j}\bo_{\kk} \cong {\SSs}''_{j}{\SSs}''_{i}\bo_{\kk}.
	\end{equation*}
	\item We have the following vanishing results
	\begin{align*}
		& {\SSs}'_{i+1}{\SSs}'_{i}{\TTt}'_{i+1}{\TTt}'_{i}\bo_{\kk} \cong 	{\SSs}'_{i+1}{\TTt}'_{i}{\TTt}'_{i+1}{\TTt}'_{i}\bo_{\kk}  \cong 	{\SSs}'_{i+1}{\SSs}'_{i}{\SSs}'_{i+1}{\TTt}'_{i}\bo_{\kk}  \cong 0, \\ 
		&{\TTt}'_{i}{\TTt}'_{i+1}{\SSs}'_{i}{\SSs}'_{i+1}\bo_{\kk} \cong 	{\TTt}'_{i}{\SSs}'_{i+1}{\SSs}'_{i}{\SSs}'_{i+1}\bo_{\kk} \cong 	{\TTt}'_{i}{\TTt}'_{i+1}{\TTt}'_{i}{\SSs}'_{i+1}\bo_{\kk} \cong 0, \\
		& {\SSs}''_{i}{\SSs}''_{i+1}{\TTt}''_{i}{\TTt}''_{i+1}\bo_{\kk} \cong 	{\SSs}''_{i}{\TTt}''_{i+1}{\TTt}''_{i}{\TTt}''_{i+1}\bo_{\kk}  \cong 	{\SSs}''_{i}{\SSs}''_{i+1}{\SSs}''_{i}{\TTt}''_{i+1}\bo_{\kk}  \cong 0, \\ 
		&{\TTt}''_{i+1}{\TTt}''_{i}{\SSs}''_{i+1}{\SSs}''_{i}\bo_{\kk} \cong 	{\TTt}''_{i+1}{\SSs}''_{i}{\SSs}''_{i+1}{\SSs}''_{i}\bo_{\kk} \cong 	{\TTt}''_{i+1}{\TTt}''_{i}{\TTt}''_{i+1}{\SSs}''_{i}\bo_{\kk} \cong 0. \\
	\end{align*}
	\item We have the following exact triangles in $\Hom(\Kk(\kk),\Kk(\kk))$. 
	\begin{align*}
		&{\TTt}'_{i}{\TTt}'_{i+1}{\TTt}'_{i}\bo_{\kk} \rightarrow {\TTt}'_{i+1}{\TTt}'_{i}{\TTt}'_{i+1}\bo_{\kk} \rightarrow {\TTt}'_{i+1}{\TTt}'_{i}{\TTt}'_{i+1}{\SSs}'_{i}\bo_{\kk}, \\	
		&{\SSs}'_{i}{\SSs}'_{i+1}{\SSs}'_{i}{\TTt}'_{i+1}\bo_{\kk} \rightarrow {\SSs}'_{i}{\SSs}'_{i+1}{\SSs}'_{i}\bo_{\kk} \rightarrow {\SSs}'_{i+1}{\SSs}'_{i}{\SSs}'_{i+1}\bo_{\kk}, \\
		&{\TTt}''_{i+1}{\TTt}''_{i}{\TTt}''_{i+1}\bo_{\kk} \rightarrow {\TTt}''_{i}{\TTt}''_{i+1}{\TTt}''_{i}\bo_{\kk} \rightarrow {\TTt}''_{i}{\TTt}''_{i+1}{\TTt}''_{i}{\SSs}''_{i+1}\bo_{\kk}, \\
		&{\SSs}''_{i+1}{\SSs}''_{i}{\SSs}''_{i+1}{\TTt}''_{i}\bo_{\kk} \rightarrow {\SSs}''_{i+1}{\SSs}''_{i}{\SSs}''_{i+1}\bo_{\kk} \rightarrow {\SSs}''_{i}{\SSs}''_{i+1}{\SSs}''_{i}\bo_{\kk}. 
	\end{align*}
\end{enumerate}
\end{theorem}  

\begin{proof}
We will keep using the categorical relations in Definition \ref{catshifted0}; when we mention the conditions, we always mean those from Definition \ref{catshifted0}.

We only show that $({\TTt}'_{i}\bo_{\kk},{\SSs}'_{i}\bo_{\kk})$ is a pair of complementary idempotent since the argument is the same for $({\TTt}''_{i}\bo_{\kk},{\SSs}''_{i}\bo_{\kk})$. From (\ref{triTk'}), we know that the pair $({\TTt}'_{i}\bo_{\kk},{\SSs}'_{i}\bo_{\kk})$ is already in a exact triangle. Thus it suffices to show that ${\TTt}'_{i}\bo_{\kk}$ and ${\SSs}'_{i}\bo_{\kk}$ are weak idempotents and ${\TTt}'_{i}{\SSs}'_{i}\bo_{\kk} \cong 0 \cong {\SSs}'_{i}{\TTt}'_{i}\bo_{\kk}$.

By condition (8)(a), we have 
\begin{align*}
({\TTt}'_{i})^2\bo_{\kk} &\cong {\E}_{i,0}{\F}_{i,k_{i+1}}({\spi}^{+}_{i})^{-1}{\E}_{i,0}{\F}_{i,k_{i+1}}({\spi}^{+}_{i})^{-1} \bo_{\kk} \\
&\cong {\E}_{i,0}{\F}_{i,k_{i+1}}{\E}_{i,-1}\bo_{\kk-\alpha_{i}}({\spi}^{+}_{i})^{-1}{\F}_{i,k_{i+1}}({\spi}^{+}_{i})^{-1}\bo_{\kk}[1]. 
\end{align*}

From condition (11)(a), we have the following exact triangle
\begin{equation*}
{\F}_{i,k_{i+1}}{\E}_{i,-1} \bo_{\kk-\alpha_{i}} \rightarrow {\E}_{i,-1}{\F}_{i,k_{i+1}}\bo_{\kk-\alpha_{i}} \rightarrow {\spi}_{i}^{+}\bo_{\kk-\alpha_{i}}, 
\end{equation*} so we obtain the following exact triangle
\begin{align*}
&({\TTt}'_{i})^2\bo_{\kk} \cong {\E}_{i,0}{\F}_{i,k_{i+1}}{\E}_{i,-1} \bo_{\kk-\alpha_{i}}({\spi}^{+}_{i})^{-1}{\F}_{i,k_{i+1}}({\spi}^{+}_{i})^{-1}\bo_{\kk}[1] \\
&\rightarrow {\E}_{i,0}{\E}_{i,-1}{\F}_{i,k_{i+1}}\bo_{\kk-\alpha_{i}}({\spi}^{+}_{i})^{-1}{\F}_{i,k_{i+1}}({\spi}^{+}_{i})^{-1}\bo_{\kk}[1] \\ 
&\rightarrow {\E}_{i,0}{\spi}_{i}^{+}\bo_{\kk-\alpha_{i}}({\spi}^{+}_{i})^{-1}{\F}_{i,k_{i+1}}({\spi}^{+}_{i})^{-1}\bo_{\kk}[1]\cong {\E}_{i,0}{\F}_{i,k_{i+1}}({\spi}^{+}_{i})^{-1}\bo_{\kk}[1]={\TTt}'_{i}\bo_{\kk}[1].
\end{align*}

By condition (6)(a), we get ${\E}_{i,0}{\E}_{i,-1}{\F}_{i,k_{i+1}}\bo_{\kk-\alpha_{i}}({\spi}^{+}_{i})^{-1}{\F}_{i,k_{i+1}}({\spi}^{+}_{i})^{-1}\bo_{\kk}[1]  \cong 0$. This implies that $({\TTt}'_{i})^2\bo_{\kk} \cong {\TTt}'_{i}\bo_{\kk}$. Similarly, by condition (9)(a), we have 
\begin{align*}
({\SSs}'_{i})^2\bo_{\kk} &\cong {\F}_{i,k_{i+1}}{\E}_{i,0}({\spi}^{+}_{i})^{-1}{\F}_{i,k_{i+1}}{\E}_{i,0}({\spi}^{+}_{i})^{-1}\bo_{\kk}[2] \\
&\cong {\F}_{i,k_{i+1}}{\E}_{i,0}{\F}_{i,k_{i+1}+1}\bo_{\kk+\alpha_{i}}({\spi}^{+}_{i})^{-1}{\E}_{i,0}({\spi}^{+}_{i})^{-1}\bo_{\kk}[1]. 
\end{align*}

From condition (11)(a), we have the following exact triangle
\begin{equation*}
{\F}_{i,k_{i+1}+1}{\E}_{i,0}\bo_{\kk+\alpha_{i}} \rightarrow {\E}_{i,0}{\F}_{i,k_{i+1}+1}\bo_{\kk+\alpha_{i}} \rightarrow {\spi}_{i}^{+}\bo_{\kk+\alpha_{i}}
\end{equation*} so we obtain the following exact triangle 
\begin{align*}
&{\F}_{i,k_{i+1}}{\F}_{i,k_{i+1}+1}{\E}_{i,0}\bo_{\kk+\alpha_{i}}({\spi}^{+}_{i})^{-1}{\E}_{i,0}({\spi}^{+}_{i})^{-1}\bo_{\kk}[1] \\ 
& \rightarrow ({\SSs}'_{i})^2\bo_{\kk} \cong {\F}_{i,k_{i+1}}{\E}_{i,0}{\F}_{i,k_{i+1}+1}\bo_{\kk+\alpha_{i}}({\spi}^{+}_{i})^{-1}{\E}_{i,0}({\spi}^{+}_{i})^{-1}\bo_{\kk}[1] \\
& \rightarrow {\F}_{i,k_{i+1}}{\spi}_{i}^{+}({\spi}^{+}_{i})^{-1}{\E}_{i,0}({\spi}^{+}_{i})^{-1}\bo_{\kk}[1] \cong {\F}_{i,k_{i+1}}{\E}_{i,0}({\spi}^{+}_{i})^{-1}\bo_{\kk}[1] = {\SSs}'_{i}\bo_{\kk}.
\end{align*}

By condition (7)(a), we get ${\F}_{i,k_{i+1}}{\F}_{i,k_{i+1}+1}{\E}_{i,0}\bo_{\kk+\alpha_{i}}({\spi}^{+}_{i})^{-1}{\E}_{i,0}({\spi}^{+}_{i})^{-1}\bo_{\kk}[1]\cong 0$, so  $({\SSs}'_{i})^2\bo_{\kk} \cong {\SSs}'_{i}\bo_{\kk}$.

Next, to prove ${\TTt}'_{i}{\SSs}'_{i}\bo_{\kk} \cong {\SSs}'_{i}{\TTt}'_{i}\bo_{\kk} \cong 0$, by conditions (7)(a) and (9)(a) we have 
\begin{align*}
{\TTt}'_{i}{\SSs}'_{i}\bo_{\kk} &\cong {\E}_{i,0}{\F}_{i,k_{i+1}}({\spi}^{+}_{i})^{-1}{\F}_{i,k_{i+1}}{\E}_{i,0}({\spi}^{+}_{i})^{-1}\bo_{\kk}[1] \\ 
&\cong {\E}_{i,0}{\F}_{i,k_{i+1}}{\F}_{i,k_{i+1}+1}({\spi}^{+}_{i})^{-1}{\E}_{i,0}({\spi}^{+}_{i})^{-1}\bo_{\kk} \cong 0,
\end{align*} similarly by conditions (6)(a) and (8)(a) we have 
\begin{align*}
{\SSs}'_{i}{\TTt}'_{i}\bo_{\kk} &\cong {\F}_{i,k_{i+1}}{\E}_{i,0}({\spi}^{+}_{i})^{-1}{\E}_{i,0}{\F}_{i,k_{i+1}}({\spi}^{+}_{i})^{-1}\bo_{\kk}[1] \\
&\cong {\F}_{i,k_{i+1}}{\E}_{i,0}{\E}_{i,-1}({\spi}^{+}_{i})^{-1}{\F}_{i,k_{i+1}}({\spi}^{+}_{i})^{-1}\bo_{\kk}[2] \cong 0.
\end{align*} 

Hence, we prove that $({\TTt}'_{i}\bo_{\kk},{\SSs}'_{i}\bo_{\kk})$ is a pair of complementary idempotents for all $1 \leq i \leq n-1$. 

Next, we show the relations (A), (B), and (C). It is clear that (A) is a direct consequence of conditions (6)(c), (7)(c), (8)(c), (9)(c), and (10).

For (B) and (C), we only prove the first one since the rest can be proved by using the same argument.

We prove $ {\SSs}'_{i+1}{\SSs}'_{i}{\TTt}'_{i+1}{\TTt}'_{i}\bo_{\kk} \cong 0$ for (B). By definition and conditions (8)(a)(b), (9)(a)(b), we have 
\begin{align}
    \begin{split}
	 &{\SSs}'_{i+1}{\SSs}'_{i}{\TTt}'_{i+1}{\TTt}'_{i}\bo_{\kk} \\
	 &\cong {\F}_{i+1,k_{i+2}}{\E}_{i+1,0}({\spi}^{+}_{i+1})^{-1}{\F}_{i,k_{i+1}}{\E}_{i,0}({\spi}^{+}_{i})^{-1}{\E}_{i+1,0}{\F}_{i+1,k_{i+2}}({\spi}^{+}_{i+1})^{-1}{\E}_{i,0}{\F}_{i,k_{i+1}}({\spi}^{+}_{i})^{-1}\bo_{\kk}[2]  \\
	 &\cong {\F}_{i+1,k_{i+2}}{\E}_{i+1,0}{\F}_{i,k_{i+1}}{\E}_{i,0}({\spi}^{+}_{i+1})^{-1}({\spi}^{+}_{i})^{-1}{\E}_{i+1,0}{\F}_{i+1,k_{i+2}}{\E}_{i,0}{\F}_{i,k_{i+1}}({\spi}^{+}_{i+1})^{-1}({\spi}^{+}_{i})^{-1}\bo_{\kk}[2]  \\
	 & \cong {\F}_{i+1,k_{i+2}}{\E}_{i+1,0}{\F}_{i,k_{i+1}}{\E}_{i,0}{\E}_{i+1,0}{\F}_{i+1,k_{i+2}}({\spi}^{+}_{i+1})^{-1}({\spi}^{+}_{i})^{-1}{\E}_{i,0}{\F}_{i,k_{i+1}}({\spi}^{+}_{i+1})^{-1}({\spi}^{+}_{i})^{-1}\bo_{\kk}[2]  \\
	 & \cong {\F}_{i+1,k_{i+2}}{\E}_{i+1,0}{\F}_{i,k_{i+1}}{\E}_{i,0}{\E}_{i+1,0}{\F}_{i+1,k_{i+2}}{\E}_{i,-1}{\F}_{i,k_{i+1}+1}({\spi}^{+}_{i+1})^{-2}({\spi}^{+}_{i})^{-2}\bo_{\kk}[2]. \label{eq 25}
	\end{split}
\end{align}

By condition (10), (\ref{eq 25}) becomes 
\begin{equation} \label{eq 26}
{\F}_{i+1,k_{i+2}}{\F}_{i,k_{i+1}}{\E}_{i+1,0}{\E}_{i,0}{\E}_{i+1,0}{\E}_{i,-1}{\F}_{i+1,k_{i+2}}{\F}_{i,k_{i+1}+1}({\spi}^{+}_{i+1})^{-2}({\spi}^{+}_{i})^{-2}\bo_{\kk}.
\end{equation}

From condition (6)(b) with $s=-1$, we have the following exact triangle
\begin{equation*}
{\E}_{i+1,-1}{\E}_{i,r+1}\bo_{\kk} \rightarrow {\E}_{i+1,0}{\E}_{i,r}\bo_{\kk} \rightarrow {\E}_{i,r}{\E}_{i+1,0}\bo_{\kk},
\end{equation*} applying ${\E}_{i+1,0}\bo_{\kk+\alpha_{i}+\alpha_{i+1}}$ and using condition (6)(a) we get 
\begin{equation*}
{\E}_{i+1,0}{\E}_{i+1,0}{\E}_{i,r}\bo_{\kk} \cong {\E}_{i+1,0}{\E}_{i,r}{\E}_{i+1,0}\bo_{\kk}.
\end{equation*} In particular, for $r=0$, we have  ${\E}_{i+1,0}{\E}_{i+1,0}{\E}_{i,0}\bo_{\kk} \cong {\E}_{i+1,0}{\E}_{i,0}{\E}_{i+1,0}\bo_{\kk}$ for all $\kk$.

Thus the term ${\E}_{i+1,0}{\E}_{i,0}{\E}_{i+1,0}{\E}_{i,-1}\bo_{\kk-\alpha_{i}-\alpha_{i+1}}$ in (\ref{eq 26}) becomes
\begin{equation*}
	{\E}_{i+1,0}{\E}_{i,0}{\E}_{i+1,0}{\E}_{i,-1}\bo_{\kk-\alpha_{i}-\alpha_{i+1}} \cong {\E}_{i+1,0}{\E}_{i+1,0}{\E}_{i,0}{\E}_{i,-1}\bo_{\kk-\alpha_{i}-\alpha_{i+1}} \cong 0 \ (\text{by condition (6)(a)})
\end{equation*} so (\ref{eq 26}) equal to zero and thus ${\SSs}'_{i+1}{\SSs}'_{i}{\TTt}'_{i+1}{\TTt}'_{i}\bo_{\kk} \cong 0$.

Using the fact that $({\TTt}'_{i}\bo_{\kk} ,{\SSs}'_{i}\bo_{\kk})$ and $({\TTt}'_{i+1}\bo_{\kk} ,{\SSs}'_{i+1}\bo_{\kk})$ are pairs of complementary idempotents, we can show that 
\begin{equation*}
{\SSs}'_{i+1}{\SSs}'_{i}{\TTt}'_{i+1}{\TTt}'_{i}\bo_{\kk} \cong 	{\SSs}'_{i+1}{\TTt}'_{i}{\TTt}'_{i+1}{\TTt}'_{i}\bo_{\kk}  \cong 	{\SSs}'_{i+1}{\SSs}'_{i}{\SSs}'_{i+1}{\TTt}'_{i}\bo_{\kk} \cong 0.  
\end{equation*}  

Finally we prove there exists exact triangle $ {\TTt}'_{i}{\TTt}'_{i+1}{\TTt}'_{i}\bo_{\kk} \rightarrow {\TTt}'_{i+1}{\TTt}'_{i}{\TTt}'_{i+1}\bo_{\kk} \rightarrow {\TTt}'_{i+1}{\TTt}'_{i}{\TTt}'_{i+1}{\SSs}'_{i}\bo_{\kk}$ for (C). Applying ${\TTt}'_{i+1}{\TTt}'_{i}{\TTt}'_{i+1}\bo_{\kk}$ to the exact triangle $ {\TTt}'_{i}\bo_{\kk} \rightarrow \bo_{\kk} \rightarrow {\SSs}'_{i}\bo_{\kk}$, we get 
\begin{equation} \label{eq 28}
	{\TTt}'_{i+1}{\TTt}'_{i}{\TTt}'_{i+1}{\TTt}'_{i}\bo_{\kk} \rightarrow {\TTt}'_{i+1}{\TTt}'_{i}{\TTt}'_{i+1}\bo_{\kk} \rightarrow {\TTt}'_{i+1}{\TTt}'_{i}{\TTt}'_{i+1}{\SSs}'_{i}\bo_{\kk} 
\end{equation}

From (B), we have ${\SSs}'_{i+1}{\TTt}'_{i}{\TTt}'_{i+1}{\TTt}'_{i}\bo_{\kk}  \cong 0$. Applying ${\TTt}'_{i}{\TTt}'_{i+1}{\TTt}'_{i}\bo_{\kk}$ to the exact triangle $ {\TTt}'_{i+1}\bo_{\kk} \rightarrow \bo_{\kk} \rightarrow {\SSs}'_{i+1}\bo_{\kk}$, we get 
\begin{equation*}
{\TTt}'_{i+1}{\TTt}'_{i}{\TTt}'_{i+1}{\TTt}'_{i}\bo_{\kk} \rightarrow {\TTt}'_{i}{\TTt}'_{i+1}{\TTt}'_{i}\bo_{\kk} \rightarrow {\SSs}'_{i+1}{\TTt}'_{i}{\TTt}'_{i+1}{\TTt}'_{i}\bo_{\kk}\cong 0,
\end{equation*} which implies that ${\TTt}'_{i+1}{\TTt}'_{i}{\TTt}'_{i+1}{\TTt}'_{i}\bo_{\kk} \cong {\TTt}'_{i}{\TTt}'_{i+1}{\TTt}'_{i}\bo_{\kk}$.

Thus (\ref{eq 28}) becomes
\begin{equation*}
	 {\TTt}'_{i}{\TTt}'_{i+1}{\TTt}'_{i}\bo_{\kk} \rightarrow {\TTt}'_{i+1}{\TTt}'_{i}{\TTt}'_{i+1}\bo_{\kk} \rightarrow {\TTt}'_{i+1}{\TTt}'_{i}{\TTt}'_{i+1}{\SSs}'_{i}\bo_{\kk} 
\end{equation*} which is our desire triangle.
\end{proof}

We state some corollaries of the main theorem. The first is a direct consequence of Proposition \ref{cohTS}.

\begin{corollary} \label{catcohTS}
For each $1 \leq i \leq n-1$, the pairs of complementary idempotents $({\TTt}'_{i}\bo_{\kk},{\SSs}'_{i}\bo_{\kk})$ and $({\TTt}''_{i}\bo_{\kk},{\SSs}''_{i}\bo_{\kk})$
in the triangulated monoidal category $\Hom(\Kk(\kk),\Kk(\kk))$ satisfy the following
\begin{align*}
\Hom({\TTt}'_{i}\bo_{\kk},{\TTt}'_{i}\bo_{\kk}) &\cong \End({\E}_{i,0}\bo_{\kk-\alpha_{i}}), \ &\Hom({\TTt}''_{i}\bo_{\kk},{\TTt}''_{i}\bo_{\kk}) &\cong \End({\F}_{i,0}\bo_{\kk+\alpha_{i}} ) \\
\Hom({\SSs}'_{i}\bo_{\kk},{\SSs}'_{i}\bo_{\kk}) &\cong \End({\F}_{i,k_{i+1}}\bo_{\kk+\alpha_i}), \ &\Hom({\SSs}''_{i}\bo_{\kk},{\SSs}''_{i}\bo_{\kk}) &\cong \End({\E}_{i,-k_{i}}\bo_{\kk-\alpha_i}) \\
\Hom({\TTt}'_{i}\bo_{\kk},{\SSs}'_{i}\bo_{\kk}) &\cong 0, \ &\Hom({\TTt}''_{i}\bo_{\kk},{\SSs}''_{i}\bo_{\kk}) &\cong 0.
\end{align*}
\end{corollary}

\begin{proof}
We prove the case for $({\TTt}'_{i}\bo_{\kk},{\SSs}'_{i}\bo_{\kk})$. By Proposition \ref{cohTS} and definitions of ${\TTt}'_{i}\bo_{\kk}$, ${\SSs}'_{i}\bo_{\kk}$ from (\ref{Tk'}), (\ref{Sk'}), using adjunctions we obtain the following
\begin{align}
\begin{split} \label{adj}
\Hom({\TTt}'_{i}\bo_{\kk},{\TTt}'_{i}\bo_{\kk}) &\cong \Hom({\TTt}'_{i}\bo_{\kk},\bo_{\kk}) \cong \Hom({\E}_{i,0}\bo_{\kk-\alpha_{i}}({\E}_{i,0}\bo_{\kk-\alpha_{i}})^{R},\bo_{\kk}) \\
&\cong \Hom({\E}_{i,0}\bo_{\kk-\alpha_{i}},{\E}_{i,0}\bo_{\kk-\alpha_{i}})=\End({\E}_{i,0}\bo_{\kk-\alpha_{i}}), \\ 
\Hom({\SSs}'_{i}\bo_{\kk}{\SSs}'_{i}\bo_{\kk}) &\cong \Hom(\bo_{\kk},{\SSs}'_{i}\bo_{\kk}) \cong \Hom(\bo_{\kk},{\F}_{i,k_{i+1}}\bo_{\kk+\alpha_i}({\F}_{i,k_{i+1}}\bo_{\kk+\alpha_i})^{L}) \\ 
&\cong \Hom({\F}_{i,k_{i+1}}\bo_{\kk+\alpha_i},{\F}_{i,k_{i+1}}\bo_{\kk+\alpha_i})= \End({\F}_{i,k_{i+1}}\bo_{\kk+\alpha_i}).
\end{split}
\end{align}

Finally, $\Hom({\TTt}'_{i}\bo_{\kk},{\SSs}'_{i}\bo_{\kk}) \cong 0 \cong \Hom({\TTt}''_{i}\bo_{\kk},{\SSs}''_{i}\bo_{\kk})$ follows from (\ref{orthTS}) in Proposition \ref{cohTS}.
\end{proof}

We mention a remark about the above corollary.

\begin{remark}
We expect that the morphisms $\epsilon, \ \eta$ and $\epsilon', \ \eta'$ in the exact triangles (\ref{triTk'}) and (\ref{TriTk''}) should come from the units and counits induced from the identity morphisms in (\ref{adj}). We would like to address this in the future since it is related to constructing a 2-representation of $\dot{\Uu}_{0,N}(L\SL_{n})$. 
\end{remark}

The second is an application to the geometric setting where the weight categories are derived categories of coherent sheaves on $n$-step partial flag varieties.

\begin{corollary} \label{mainthmgeo}
For each $\kk \in C(n,N)$, consider the triangulated monoidal category $\Dd^b(\Fl_{\kk}(\CC^N) \times \Fl_{\kk}(\CC^N))$, where the convolution product gives the monoidal structure. Then we have the following two families of pairs of complementary idempotents that are given by FM kernels.
\begin{align*}
\{({\Tt}'_{i}\bo_{\kk} \coloneqq {\Ee}_{i,0}\ast {\Ff}_{i,k_{i+1}} \ast ({\Psi}^{+}_{i})^{-1} \bo_{\kk}, \ {\Ss}'_{i}\bo_{\kk} \coloneqq {\Ff}_{i,k_{i+1}} \ast {\Ee}_{i,0} \ast ({\Psi}^{+}_{i})^{-1} \bo_{\kk}[1])\}_{1 \leq i \leq n-1} \\
\{({\Tt}''_{i}\bo_{\kk} \coloneqq {\Ff}_{i,0}\ast {\Ee}_{i,-k_{i}} \ast ({\Psi}^{-}_{i})^{-1}\bo_{\kk}, \ {\Ss}''_{i}\bo_{\kk} \coloneqq {\Ee}_{i,-k_{i}} \ast {\Ff}_{i,0}\ast  ({\Psi}^{-}_{i})^{-1}\bo_{\kk}[1])\}_{1 \leq i \leq n-1}.
\end{align*} Moreover, these FM kernels satisfy relations (A), (B), and (C) in Theorem \ref{Theorem 5}. The cohomology of ${\Tt}'_{i}\bo_{\kk}$, ${\Tt}''_{i}\bo_{\kk}$, ${\Ss}'_{i}\bo_{\kk}$, and ${\Ss}''_{i}\bo_{\kk}$ are given by
\begin{align}
&\Hom({\Tt}'_{i}\bo_{\kk},{\Tt}'_{i}\bo_{\kk}) \cong \bigoplus_{k} H^*(W^{1}_{i}(\kk-\alpha_i),\bigwedge^k\Nn_{W^{1}_{i}(\kk-\alpha_i)}) \cong \Hom({\Ss}''_{i}\bo_{\kk},{\Ss}''_{i}\bo_{\kk}), \label{cohT'k}\\
&\Hom({\Tt}''_{i}\bo_{\kk},{\Tt}''_{i}\bo_{\kk}) \cong \bigoplus_{k} H^*(W^{1}_{i}(\kk),\bigwedge^k\Nn_{W^{1}_{i}(\kk)}) \cong \Hom({\Ss}'_{i}\bo_{\kk},{\Ss}'_{i}\bo_{\kk}),\label{cohT''k}
\end{align} where $W^1_{i}(\kk)$ is defined in (\ref{eq 30}) and $\Nn_{W^1_{i}(\kk)}$ denote its normal bundle in $\Fl_{\kk}(\CC^N) \times \Fl_{\kk+\alpha_i}(\CC^N)$.
\end{corollary}

\begin{proof}
The only things we have to prove are the cohomologies (\ref{cohT'k}) and (\ref{cohT''k}). 

We prove (\ref{cohT'k}) only since the argument for (\ref{cohT''k}) is the same. By Corollary \ref{catcohTS}, we know that $\Hom({\Tt}'_{i}\bo_{\kk},{\Tt}'_{i}\bo_{\kk}) \cong \End(\Ee_{i,0}\bo_{\kk-\alpha_i})$. From the definition of $\Ee_{i,0}\bo_{\kk-\alpha_i}$ in Proposition \ref{Thm catact}, we obtain
\begin{align*}
&\End(\Ee_{i,0}\bo_{\kk-\alpha_i}) \\
&=\Hom_{\Fl_{\kk-\alpha_i}(\CC^N) \times \Fl_{\kk}(\CC^N)}(\iota_{*}\Oo_{W^{1}_{i}(\kk-\alpha_i)}, \iota_*\Oo_{W^{1}_{i}(\kk-\alpha_i)}) 
\cong  \Hom_{W^{1}_{i}(\kk-\alpha_i)}(\iota^*\iota_{*}\Oo_{W^{1}_{i}(\kk-\alpha_i)}, \Oo_{W^{1}_{i}(\kk-\alpha_i)}) \\
& \cong \Hom_{W^{1}_{i}(\kk-\alpha_i)}(\bigoplus_{k}\bigwedge^k\Nn^{\vee}_{W^{1}_{i}(\kk-\alpha_i)}[k], \Oo_{W^{1}_{i}(\kk-\alpha_i)})  \cong \bigoplus_{k} H^*(W^{1}_{i}(\kk-\alpha_i),\bigwedge^k\Nn_{W^{1}_{i}(\kk-\alpha_i)})
\end{align*} where $\iota:W^{1}_{i}(\kk-\alpha_i) \rightarrow \Fl_{\kk-\alpha_i}(\CC^N) \times \Fl_{\kk}(\CC^N)$ is the natural inclusion.

On the other hand, from Corollary \ref{catcohTS}, using condition (8)(a), we have 
\begin{align*}
\Hom({\Ss}''_{i}\bo_{\kk},{\Ss}''_{i}\bo_{\kk}) &\cong \Hom(\Ee_{i,-k_i}\bo_{\kk-\alpha_i},\Ee_{i,-k_i}\bo_{\kk-\alpha_i}) \\
&\cong \Hom((\Psi^{+})^{-k_i}\Ee_{i,0}(\Psi^{+})^{k_i}\bo_{\kk-\alpha_i}[-k_i],(\Psi^{+})^{-k_i}\Ee_{i,0}(\Psi^{+})^{k_i}\bo_{\kk-\alpha_i}[-k_i]) \\
& \cong \Hom(\Ee_{i,0}\bo_{\kk-\alpha_i},\Ee_{i,0}\bo_{\kk-\alpha_i})=\End(\Ee_{i,0}\bo_{\kk-\alpha_i}).
\end{align*}
\end{proof}

The final one is an application to prove the categorical braid relations in (\ref{6.2}) in Proposition \ref{catactaffine0Hecke}. Note that although we already prove (\ref{6.2}) in \cite[Section 6]{Hsu}, it still takes one page to write down the proof there, and the following proof is much shorter.

\begin{corollary} \label{shortcatbraid}
In the case where $\kk=(1,1,...,1)$, we have $\Fl_{\kk}(\CC^N)=G/B$.  Consider $\Tt_{i} \coloneqq \Oo_{G/B \times_{G/P_i} G/B} \in \Dd^b(G/B \times G/B)$  where $G/B \times_{G/P_i} G/B$ is the Bott-Samelson variety for all $1 \leq i \leq N-1$. Then we have the braid relation $\Tt_{i}\ast \Tt_{i+1} \ast \Tt_{i} \cong \Tt_{i+1} \ast \Tt_{i} \ast \Tt_{i+1}$.  
\end{corollary}

\begin{proof}
For the full flag variety case $G/B=Fl_{(1,1,...,1)}(\CC^N)$, from (\ref{eq p2}) we have 
\begin{equation*}
{\Tt}'_{i}\bo_{(1,1,...,1)}={\Ee}_{i,0}\ast {\Ff}_{i,1} \ast ({\Psi}^{+}_{i})^{-1} \bo_{(1,1,...,1)}={\Ff}_{i,0}\ast {\Ee}_{i,1} \ast ({\Psi}^{-}_{i})^{-1}\bo_{(1,1,...,1)}={\Tt}''_{i}\bo_{(1,1,...,1)}.
\end{equation*}

By Corollary \ref{mainthmgeo}, we know that $({\Tt}'_{i}\bo_{(1,1,...,1)},{\Ss}'_{i}\bo_{(1,1,...,1)})$ and $({\Tt}''_{i}\bo_{(1,1,...,1)},{\Ss}''_{i}\bo_{(1,1,...,1)})$ are pairs of complementary idempotents in $\Dd^b(G/B \times G/B)$. Since ${\Tt}'_{i}\bo_{(1,1,...,1)}={\Tt}''_{i}\bo_{(1,1,...,1)}$, we obtain that ${\Ss}'_{i}\bo_{(1,1,...,1)}={\Ss}''_{i}\bo_{(1,1,...,1)}$. 

Thus by (B) in Theorem \ref{Theorem 5} and Corollary \ref{mainthmgeo}, we obtain
\begin{align*}
&{\Tt}'_{i+1}\ast{\Tt}'_{i}\ast{\Tt}'_{i+1}\ast{\Ss}'_{i}\bo_{(1,1,...,1)}={\Tt}''_{i+1}\ast{\Tt}''_{i}\ast{\Tt}''_{i+1}\ast{\Ss}''_{i}\bo_{(1,1,...,1)}=0, \\
&{\Tt}''_{i}\ast{\Tt}''_{i+1}\ast{\Tt}''_{i}\ast{\Ss}''_{i+1}\bo_{(1,1,...,1)}={\Tt}'_{i}\ast{\Tt}'_{i+1}\ast{\Tt}'_{i}\ast{\Ss}'_{i+1}\bo_{(1,1,....,1)}=0.
\end{align*} 

The above tells us that the correction terms (or the third terms) of the first and third exact triangles in (C) from Theorem \ref{Theorem 5} are zero, which implies the braid relations.
\end{proof}

\subsection{Miscellany and generalizations}

In this subsection, we deduce some related results from the discussions and the main theorem in the previous subsection.

Note that in the shifted 0-affine algebra $\dot{\brmU}_{0,N}(L\SL_n)$, the loop-like generators $e_{i,r}1_{\kk}$ and $f_{i,s}1_{\kk}$ are defined within the ranges $-k_i-1 \leq r \leq 0$ and $0 \leq s \leq k_{i+1}+1$, respectively. Thus from the definition of ${\TTt}'_{i}\bo_{\kk}$, ${\TTt}''_{i}\bo_{\kk}$  in (\ref{Tk'}), (\ref{Tk''}), it is also natural to consider the compositions ${\E}_{i,r}\bo_{\kk-\alpha_{i}}({\E}_{i,r}\bo_{\kk-\alpha_{i}})^{R}$ for $-k_i-2 \leq r \leq 0$.

More precisely, we can define 
\begin{align}
   &{\TTt}'_{i,r}\bo_{\kk} \coloneqq  {\E}_{i,r}\bo_{\kk-\alpha_{i}}({\E}_{i,r}\bo_{\kk-\alpha_{i}})^{R} \label{Tk'r}\\ 
   &{\TTt}''_{i,s}\bo_{\kk} \coloneqq  {\F}_{i,s}\bo_{\kk+\alpha_{i}}({\F}_{i,s}\bo_{\kk+\alpha_{i}})^{R} \label{Tk''s}
\end{align} for all $-k_i-2 \leq r \leq 0$, $0 \leq s \leq k_{i+1}+2$ such that ${\TTt}'_{i,0}\bo_{\kk} ={\TTt}'_{i}\bo_{\kk}$, ${\TTt}''_{i,0}\bo_{\kk} ={\TTt}''_{i}\bo_{\kk}$.

Using the adjunction condition (4), condition (8)(a), and condition (9)(a) from Definition \ref{catshifted0}, we obtain that 
\begin{align}
\begin{split} \label{T'kriso}
    {\TTt}'_{i,r}\bo_{\kk}&= {\E}_{i,r}\bo_{\kk-\alpha_{i}}({\E}_{i,r}\bo_{\kk-\alpha_{i}})^{R}\cong {\E}_{i,r}({\spi}^{+}_{i})^{r+1}{\F}_{i,k_{i+1}+1}({\spi}^{+}_{i})^{-r-2}\bo_{\alpha_{i}}[-r-1] \\
   &\cong ({\spi}^{+}_{i})^{r+1}{\E}_{i,-1}{\F}_{i,k_{i+1}+1}({\spi}^{+}_{i})^{-r-2}\bo_{\alpha_{i}} \cong ({\spi}^{+}_{i})^{r}{\E}_{i,0}{\spi}^{+}_{i}{\F}_{i,k_{i+1}+1}({\spi}^{+}_{i})^{-r-2}\bo_{\alpha_{i}}[-1] \\
   &\cong ({\spi}^{+}_{i})^{r}{\E}_{i,0}{\F}_{i,k_{i+1}}({\spi}^{+}_{i})^{-r-1}\bo_{\alpha_{i}} \cong ({\spi}^{+}_{i})^{r} {\TTt}'_{i}({\spi}^{+}_{i})^{-r}\bo_{\kk}.
   \end{split}
\end{align} Similarly, we obtain
\begin{equation} \label{T''ksiso}
    {\TTt}''_{i,s}\bo_{\kk} \cong ({\spi}^{-}_{i})^{-s} {\TTt}''_{i}({\spi}^{-}_{i})^{s}\bo_{\kk}.
\end{equation}

On the other hand, we define
\begin{align}
&{\SSs}'_{i,s}\bo_{\kk} \coloneqq  {\F}_{i,s}\bo_{\kk+\alpha_{i}}({\F}_{i,s}\bo_{\kk+\alpha_{i}})^{L}, \label{Sk's}\\ 
&{\SSs}''_{i,r}\bo_{\kk} \coloneqq  {\E}_{i,r}\bo_{\kk-\alpha_{i}}({\E}_{i,r}\bo_{\kk-\alpha_{i}})^{L}. \label{Sk''r}    
\end{align} where $0 \leq s \leq k_{i+1}+2$ and $-k_{i}-2 \leq r \leq 0$. Then a similar calculation shows that
\begin{align}
\begin{split} \label{S'ksiso}
{\SSs}'_{i,s}\bo_{\kk}&= {\F}_{i,s}\bo_{\kk+\alpha_{i}}({\F}_{i,s}\bo_{\kk+\alpha_{i}})^{L} \cong {\F}_{i,s}({\spi}^{+}_{i})^{-s+k_{i+1}}{\E}_{i,0}({\spi}^{+}_{i})^{-s-k_{i+1}-1}\bo_{\kk+\alpha_{i}}[-s+k_{i+1}+1] \\
& \cong ({\spi}^{+}_{i})^{-s+k_{i+1}}{\F}_{i,k_{i+1}}{\E}_{i,0}({\spi}^{+}_{i})^{-s-k_{i+1}-1}\bo_{\kk+\alpha_{i}}[1]=({\spi}^{+}_{i})^{-s+k_{i+1}}{\SSs}'_{i}({\spi}^{+}_{i})^{s-k_{i+1}}\bo_{\kk}
\end{split}
\end{align} and 
\begin{equation} \label{S''kriso}
{\SSs}''_{i,r}\bo_{\kk} \cong  ({\spi}^{-}_{i})^{r+k_{i}}{\SSs}''_{i}({\spi}^{-}_{i})^{-r-k_{i}}\bo_{\kk}.  
\end{equation}

In conclusion, we obtain the following corollary.

\begin{corollary} \label{catcomm}
For $\kk \in C(n,N)$ and $r, s \in \ZZ$ such that $r+s=k_{i+1}$, we have the following families of categorical idempotents
\begin{equation*}
({\TTt}'_{i,r}\bo_{\kk}, {\SSs}'_{i,s}\bo_{\kk})= ({\E}_{i,r}\bo_{\kk-\alpha_{i}}({\E}_{i,r}\bo_{\kk-\alpha_{i}})^{R}, {\F}_{i,s}\bo_{\kk+\alpha_{i}}({\F}_{i,s}\bo_{\kk+\alpha_{i}})^{L}) 
\end{equation*} for $1 \leq i\leq n-1$. If $r, s \in \ZZ$ such that $r+s=-k_{i}$, we also have the following families of categorical idempotents
\begin{equation*}
({\TTt}''_{i,s}\bo_{\kk}, {\SSs}''_{i,r}\bo_{\kk})= ({\F}_{i,s}\bo_{\kk+\alpha_{i}}({\F}_{i,s}\bo_{\kk+\alpha_{i}})^{R}, {\E}_{i,r}\bo_{\kk-\alpha_{i}}({\E}_{i,r}\bo_{\kk-\alpha_{i}})^{L}) 
\end{equation*} for $1 \leq i\leq n-1$.  They also satisfy the same relations (A), (B), and (C) in Theorem \ref{Theorem 5}. For example, when $r+s=k_{i+1}$, we have 
\begin{align*}
&{\TTt}'_{i,r}{\TTt}'_{j,r}\bo_{\kk} \cong {\TTt}'_{j,r}{\TTt}'_{i,r}\bo_{\kk} \ \text{for} \ |i-j| \geq 2, \\
&{\SSs}'_{i+1,s}{\SSs}'_{i,s}{\TTt}'_{i+1,r}{\TTt}'_{i,r}\bo_{\kk} \cong {\SSs}'_{i+1,s}{\SSs}'_{i,s}{\SSs}'_{i+1,s}{\TTt}'_{i,r}\bo_{\kk} \cong {\SSs}'_{i+1,s}{\TTt}'_{i,r}{\TTt}'_{i+1,r}{\TTt}'_{i,r}\bo_{\kk} \cong 0,
\end{align*} and the exact triangle
\begin{equation*}
{\TTt}'_{i,r}{\TTt}'_{i+1,r}{\TTt}'_{i,r}\bo_{\kk} \rightarrow {\TTt}'_{i+1,r}{\TTt}'_{i,r}{\TTt}'_{i+1,r}\bo_{\kk} \rightarrow {\TTt}'_{i+1,r}{\TTt}'_{i,r}{\TTt}'_{i+1,r}{\SSs}'_{i,s}\bo_{\kk}.
\end{equation*} Finally, the cohomologies of ${\TTt}'_{i,r}\bo_{\kk}$, ${\SSs}'_{i,s}\bo_{\kk}$, ${\TTt}''_{i,r}\bo_{\kk}$, and ${\SSs}''_{i,s}\bo_{\kk}$ are isomorphic to the cohomologies of  ${\TTt}'_{i}\bo_{\kk}$, ${\SSs}'_{i}\bo_{\kk}$, ${\TTt}''_{i}\bo_{\kk}$, and ${\SSs}''_{i}\bo_{\kk}$ respectively.
\end{corollary}

\begin{proof}
We only prove the case where $r+s=k_{i+1}$. 

By Theorem \ref{Theorem 5} we know that $({\TTt}'_{i}\bo_{\kk},{\SSs}'_{i}\bo_{\kk})$ is a categorical idempotents. From (\ref{T'kriso}) we obtain that 
\begin{align*}
  {\TTt}'_{i,r}{\TTt}'_{i,r}\bo_{\kk} &\cong ({\spi}^{+}_{i})^{r} {\TTt}'_{i}({\spi}^{+}_{i})^{-r}\bo_{\kk}  ({\spi}^{+}_{i})^{r} {\TTt}'_{i}({\spi}^{+}_{i})^{-r}\bo_{\kk}  \\
  &= ({\spi}^{+}_{i})^{r} {\TTt}'_{i}{\TTt}'_{i}({\spi}^{+}_{i})^{-r}\bo_{\kk} \cong ({\spi}^{+}_{i})^{r} {\TTt}'_{i}({\spi}^{+}_{i})^{-r}\bo_{\kk} \cong {\TTt}'_{i,r}\bo_{\kk}, 
\end{align*} and similarly ${\SSs}'_{i,s}{\SSs}'_{i,}\bo_{\kk} \cong {\SSs}'_{i,}\bo_{\kk}$. Thus ${\TTt}'_{i,r}\bo_{\kk}$ and ${\SSs}'_{i,}\bo_{\kk}$ are weak idempotents. 

Next, using $r+s=k_{i+1}$ we have 
\begin{equation*}
{\TTt}'_{i,r}{\SSs}'_{i,s}\bo_{\kk} \cong  ({\spi}^{+}_{i})^{r} {\TTt}'_{i}({\spi}^{+}_{i})^{-r} ({\spi}^{+}_{i})^{-s+k_{i+1}}{\SSs}'_{i}({\spi}^{+}_{i})^{s-k_{i+1}}\bo_{\kk} \cong ({\spi}^{+}_{i})^{r} {\TTt}'_{i}{\SSs}'_{i}({\spi}^{+}_{i})^{s-k_{i+1}}\bo_{\kk} \cong 0
\end{equation*} and ${\SSs}'_{i,s}{\TTt}'_{i,r}\bo_{\kk} \cong 0$ similarly. Thus we prove that $({\TTt}'_{i,r}\bo_{\kk}, {\SSs}'_{i,s}\bo_{\kk})$ is a categorical idempotents.

Since ${\TTt}'_{i,r}\bo_{\kk}$ and ${\SSs}'_{i,s}\bo_{\kk}$ are isomorphic to conjugations of ${\TTt}'_{i}\bo_{\kk}$ and ${\SSs}'_{i}\bo_{\kk}$ by $({\spi}^{+})^{r}$ and $({\spi}^{+})^{-s+k_{i+1}}$ from (\ref{T'kriso}) and (\ref{S'ksiso}), respectively. For the proof of relations (A), (B), and (C) for $({\TTt}'_{i,r}\bo_{\kk}, {\SSs}'_{i,s}\bo_{\kk})$, we simply apply the conjugation by ${\spi}^{+}$ to relations (A), (B), and (C) for $({\TTt}'_{i}\bo_{\kk}, {\SSs}'_{i}\bo_{\kk})$ in Theorem \ref{Theorem 5}, and similarly for their cohomologies.
\end{proof}

\begin{remark} \label{catcomm'} 
In fact, conditions (11)(a) and (11)(b) in Definition \ref{catshifted0} are equivalent to the fact that $({\TTt}'_{i,r}\bo_{\kk}, {\SSs}'_{i,s}\bo_{\kk})$ and $({\TTt}''_{i,s}\bo_{\kk}, {\SSs}''_{i,r}\bo_{\kk})$ are categorical idempotents, respectively.
\end{remark}

\section{Application to semiorthogonal decomposition} \label{Section 6}

In this section, we apply the main results, in particular Corollary \ref{mainthmgeo}, to provide examples for Kuznetsov's result in Proposition \ref{sodfm}. 

Observe that $\Dd^b(\Fl_{\kk}(\CC^N)\times \Fl_{\kk}(\CC^N))$ acts on $\Dd^b(\Fl_{\kk}(\CC^N)$ by FM transforms. By Example \ref{ssod}, the pairs of complementary idempotents in Corollary \ref{mainthmgeo} give the following $\SOD$s of $\Dd^b(\Fl_{\kk}(\CC^N))$
\begin{align}
    \Dd^b(\Fl_{\kk}(\CC^N)) &= \langle \text{Im}\Phi_{{\Ss}'_{i}\bo_{\kk}}, \text{Im}\Phi_{{\Tt}'_{i}\bo_{\kk}} \rangle, \label{sod1}\\
    &= \langle \text{Im}\Phi_{{\Ss}''_{i}\bo_{\kk}}, \text{Im}\Phi_{{\Tt}''_{i}\bo_{\kk}} \rangle, \label{sod2}
\end{align} for all $1 \leq i \leq n-1$, where we use the notation $\text{Im}\Phi_{{\Tt}'_{i}\bo_{\kk}}$ to denote the full subcategory generated by the essential images of $\Phi_{{\Tt}'_{i}\bo_{\kk}}$. 

It would be interesting to understand those component categories in (\ref{sod1}) and (\ref{sod2}) explicitly. For example, what are the generators of those component categories?

In the rest of this article, we address the above question in the special case where $n=2$. In this case $\kk=(k_1,k_2)$ with $k_1+k_2=N$, so $\Fl_{\kk}(\CC^N)$ is the Grassmannian. Since $k_2=N-k_1$, we will use the notation $(k,N-k)$ and $\Gr(k,N)$ instead of $(k_1,N-k_1)$ and $\Fl_{\kk}(\CC^N)$.

\subsection{The Kapranov exceptional collection}
In this section, we recall the exceptional collection constructed by Kapranov \cite{Ka1} for $\Dd^b(\Gr(k,N))$. 

Let $\blam=(\lambda_{1},...,\lambda_{n})$ be a non-increasing sequence of positive integers. We can represent $\blam$ as a Young diagram with $n$ rows, aligned on the left, such that the $i$th row
has exactly $\lambda_{i}$ cells. The size of $\blam$, denoted by $|\blam|$, is the number $|\blam|=\sum_{i=1}^{n} \lambda_{i}$. The transpose diagram $\blam^*$ is obtained by exchanging rows and columns of $\blam$.

For a Young diagram $\blam$, we define the notion of its associated Schur functor.  For more details about Schur functors, we refer the readers to \cite[Chapter 4, Chapter 6]{FH}.

\begin{definition}
	Let $n \geq 1$ be a positive integer and $\blam=(\lambda_{1},...,\lambda_{n})$ be a sequence of non-increasing positive integers. The \textit{Schur functor} $\s_{\blam}$ associated to $\blam$ is defined as a functor 
	\[
	\s_{\blam}:\text{Vect}_{\CC} \rightarrow \text{Vect}_{\CC}
	\] 
	such that for any vector space $V$, $\s_{\blam}V$ coincides with the image of the Young symmetrizer $c_{\blam}$ in the space of tensors of $V$ of rank $n$: i.e., $\s_{\blam}V=\text{Im} (c_{\blam}|_{V^{\otimes n}})$.
\end{definition}

First, for the case where $n=2$, they are just the Grassmannians
\[
\Gr(k,N)=\{0 \subset V \subset \CC^N \ | \ \tdim_{\CC} V=k\}
\] of $k$-dimensional subspaces in $\CC^N$.

Denote $\V$ to be the tautological rank $k$ bundle on $\Gr(k,N)$ and $\CC^N/\V$ to be the tautological rank $N-k$ quotient bundle.  For non-negative integers $a, \ b \geq 0$, we denote by $P(a,b)$ the set of Young diagrams $\blam$ such that $\lambda_{1} \leq a$ and $\lambda_{b+1} =0$. Then we have the following result.

\begin{proposition}[\cite{Ka1}]  \label{exccoll}  
	The following collection of sheaves 
\begin{equation} \label{kapexc}
R_{(k,N-k)}=\{\s_{\blam}\V\ | \ \blam \in P(N-k,k) \} 
\end{equation} is a strong full exceptional collection in $\Dd^b(\Gr(k,N))$.  Its dual exceptional collection is given by 
\begin{equation} \label{kapdualexc}
R'_{(k,N-k)} = \{\s_{\bmu}\CC^N/\V[-|\bmu|] \ | \ \bmu \in P(k,N-k) \}.
\end{equation}
\end{proposition}

Since an exceptional collection leads to a $\SOD$, we obtain the following $\SOD$s of $\Dd^b(\Gr(k,N))$
\begin{align*}
    \Dd^b(\Gr(k,N))&=\langle \s_{\blam}\V \rangle_{\blam \in P(N-k,k)},  \\
    &=\langle \s_{\bmu}\CC^N/\V[-|\bmu|] \rangle_{\bmu \in P(k,N-k)}.
\end{align*}

\subsection{Calculation of the action}

Since $n=2$, we only have $i=1$. So we will ignore the subscript $i$ for simplicity. Then
we only have two pairs of complementary idempotents in $\Dd^b(\Gr(k,N) \times \Gr(k,N))$, they are 
\begin{align*}
({\Tt}'\bo_{(k,N-k)} \coloneqq {\Ee}_{0}\ast {\Ff}_{N-k} \ast ({\Psi}^{+})^{-1} \bo_{(k,N-k)}, \ {\Ss}'\bo_{(k,N-k)} \coloneqq {\Ff}_{N-k} \ast {\Ee}_{0} \ast ({\Psi}^{+})^{-1} \bo_{(k,N-k)}[1]), \\
({\Tt}''\bo_{(k,N-k)} \coloneqq {\Ff}_{0}\ast {\Ee}_{-k} \ast ({\Psi}^{-})^{-1}\bo_{(k,N-k)}, \ {\Ss}''\bo_{(k,N-k)} \coloneqq {\Ee}_{-k} \ast {\Ff}_{0}\ast  ({\Psi}^{-})^{-1}\bo_{(k,N-k)}[1]),
\end{align*} and thus we have the following $\SOD$s of $\Dd^b(\Gr(k,N))$
\begin{align}
\Dd^b(\Gr(k,N)) &= \langle \text{Im}\Phi_{{\Ss}'\bo_{(k,N-k)}}, \text{Im}\Phi_{{\Tt}'\bo_{(k,N-k)}} \rangle,  \label{sodk1} \\
&= \langle \text{Im}\Phi_{{\Ss}''\bo_{(k,N-k)}}, \text{Im}\Phi_{{\Tt}''\bo_{(k,N-k)}} \rangle, \label{sodk2}
\end{align}

In this subsection, we calculate the action of the FM transforms $\Phi_{{\Tt}'\bo_{(k,N-k)}}$ and $\Phi_{{\Tt}''\bo_{(k,N-k)}}$ on the exceptional collections in Theorem \ref{exccoll}, and by property of complementary idempotents, the action of $\Phi_{{\Ss}'\bo_{(k,N-k)}}$ and $\Phi_{{\Ss}''\bo_{(k,N-k)}}$ will be determined by $\Phi_{{\Tt}'\bo_{(k,N-k)}}$ and $\Phi_{{\Tt}''\bo_{(k,N-k)}}$ respectively.

The key observation we use is to interpret the Kapranov exceptional collection in terms of the categorical $\dot{\brmU}_{0,N}(L\SL_2)$ action via convolution of FM kernels. More precisely, by using the Borel-Weil-Bott Theorem, in the Grassmannian case, we have 
\begin{align}
    & \s_{\blam}\V \cong \Ff_{\lambda_{1}} \ast ... \ast \Ff_{\lambda_{k}} \bo_{(0,N)} \in \Dd^b(\Gr(0,N) \times \Gr(k,N)),  \label{bwbkap1} \\
    & \s_{\bmu}(\CC^N/\V)^{\vee}\cong \Ee_{-\mu_{1}} \ast ... \ast \Ee_{-\mu_{N-k}} \bo_{(N,0)} \in \Dd^b(\Gr(N,N) \times \Gr(k,N)), \label{bwbkap2}
\end{align} where $\blam=(\lambda_{1},...,\lambda_{k}) \in P(N-k,k)$ and  $\bmu=(\mu_{1},...,\mu_{N-k}) \in P(k,N-k)$.

More precisely, we calculate $\Phi_{{\Tt}'\bo_{(k,N-k)}}$ on the exceptional collection (\ref{kapexc}) and $\Phi_{{\Tt}'\bo_{(k,N-k)}}$ on the dual (as vector bundles) of (\ref{kapdualexc}). Then we have the following result.

\begin{theorem} \label{actofidem}
For $\blam=(\lambda_{1},...,\lambda_{k}) \in P(N-k,k)$ and $\bmu=(\mu_{1},...\mu_{N-k}) \in P(k,N-k)$ we have 
\begin{align*}
\Phi_{{\Tt}'\bo_{(k,N-k)}}(\s_{\blam}\V)&=\begin{cases}
	0 & \text{if} \ \lambda_{1}=N-k \\
	\s_{\blam}\V & \text{if} \ 0 \leq \lambda_{1} \leq N-k-1,
\end{cases} \\
\Phi_{{\Tt}''\bo_{(k,N-k)}}(\s_{\bmu}(\CC^N/\V)^{\vee})&=\begin{cases}
	0 & \text{if} \ \mu_{1}=k \\
	\s_{\bmu}(\CC^N/\V)^{\vee} & \text{if} \ 0 \leq \mu_{1} \leq k-1.
\end{cases}
\end{align*} 
\end{theorem}

\begin{proof}
We only prove the case for $\Phi_{{\Tt}'\bo_{(k,N-k)}}$ since the argument for the action of $\Phi_{{\Tt}''\bo_{(k,N-k)}}$ is the same. To calculate $\Phi_{{\Tt}'\bo_{(k,N-k)}}(\s_{\blam}\V)$, we use the language of categorical $\dot{\brmU}_{0,N}(L\SL_2)$ action and (\ref{bwbkap1}) and (\ref{bwbkap2}). 

More precisely, by Proposition \ref{Thm catact}, there is a categorical $\dot{\brmU}_{0,N}(L\SL_2)$ action on $\bigoplus_{k} \Dd^b(\Gr(k,N))$. Since ${\Tt}'\bo_{(k,N-k)} \coloneqq {\Ee}_{0}\ast {\Ff}_{N-k} \ast ({\Psi}^{+})^{-1} \bo_{(k,N-k)} \in \Dd^b(\Gr(k,N) \times \Gr(k,N))$ is the FM kernel of $\Phi_{{\Tt}'\bo_{(k,N-k)}}$ and ${\E}_{0}{\F}_{N-k}({\spi}^{+})^{-1}\bo_{(k,N-k)}$, we have $\Phi_{{\Tt}'\bo_{(k,N-k)}} \cong {\E}_{0}{\F}_{N-k}{\spi}^{+}\bo_{(k,N-k)}$. On the other hand, by (\ref{bwbkap1}), we have $\s_{\blam}\V \cong \Ff_{\lambda_{1}} \ast ... \ast \Ff_{\lambda_{k}} \bo_{(0,N)}$ which is a FM kernel for the functor ${\F}_{\lambda_1}....{\F}_{\lambda_{k}}\bo_{(0,N)}$.

Thus calculate $\Phi_{{\Tt}'\bo_{(k,N-k)}}(\s_{\blam}\V)$ is equivalent to calculate ${\E}_{0}{\F}_{N-k}({\spi}^{+})^{-1}{\F}_{\lambda_1}....{\F}_{\lambda_{k}}\bo_{(0,N)}$.

By condition (9)(a), we have 
\begin{equation} \label{comp} 
{\E}_{0}{\F}_{N-k}({\spi}^{+})^{-1}{\F}_{\lambda_1}....{\F}_{\lambda_{k}}\bo_{(0,N)}  \cong {\E}_{0}{\F}_{N-k}{\F}_{\lambda_1+1}....{\F}_{\lambda_{k}+1}({\spi}^{+})^{-1}\bo_{(0,N)}[-k].  
\end{equation} 

Since $\blam \in P(N-k,k)$, we have $0 \leq \lambda_k \leq ... \leq \lambda_1 \leq N-k$. There are two cases.

If $\lambda_1=N-k$, then by condition (7)(a), we have $(\ref{comp})=0$ which prove the first case.

If $0 \leq \lambda_1 \leq N-k-1$, then by condition (11)(a) we have the following exact triangle
\begin{equation*}
    {\F}_{N-k}{\E}_{0}\bo_{(k,N-k)} \rightarrow {\E}_{0}{\F}_{N-k}\bo_{(k,N-k)} \rightarrow {\spi}^{+}\bo_{(k,N-k)}.
\end{equation*} Then we apply ${\F}_{\lambda_1+1}....{\F}_{\lambda_{k}+1}({\spi}^{+})^{-1}\bo_{(0,N)}[-k]$ to it so that we obtain the following exact triangle
\begin{align}
\begin{split} \label{comptri} 
    &{\F}_{N-k}{\E}_{0}{\F}_{\lambda_1+1}....{\F}_{\lambda_{k}+1}({\spi}^{+})^{-1}\bo_{(0,N)}[-k] 
    \rightarrow {\E}_{0}{\F}_{N-k}{\F}_{\lambda_1+1}....{\F}_{\lambda_{k}+1}({\spi}^{+})^{-1}\bo_{(0,N)}[-k] \\ &\rightarrow {\spi}^{+}{\F}_{\lambda_1+1}....{\F}_{\lambda_{k}+1}({\spi}^{+})^{-1}\bo_{(0,N)}[-k] \cong {\spi}^{+}{\F}_{\lambda_1+1}....{\F}_{\lambda_{k}+1}({\spi}^{+})^{-1}\bo_{(0,N)}[-k],
    \end{split}
\end{align} and the middle term is what we want to know.

Since $1 \leq \lambda_{k}+1 \leq ... \leq \lambda_{1}+1 \leq N-k$, by condition (11)(c) we get 
\begin{align*}
 {\E}_{0}{\F}_{\lambda_1+1}....{\F}_{\lambda_{k}+1}({\spi}^{+})^{-1}\bo_{(0,N)} &\cong {\F}_{\lambda_1+1}{\E}_{0}....{\F}_{\lambda_{k}+1}({\spi}^{+})^{-1}\bo_{(0,N)}  \\
 &\cong ... \\
 &\cong {\F}_{\lambda_1+1}....{\F}_{\lambda_{k}+1}{\E}_{0}({\spi}^{+})^{-1}\bo_{(0,N)} \cong 0.
\end{align*} On the other hand, by condition (9)(a) we get
\begin{equation*}
 {\spi}^{+}{\F}_{\lambda_1+1}....{\F}_{\lambda_{k}+1}({\spi}^{+})^{-1}\bo_{(0,N)}[-k] \cong {\F}_{\lambda_1}....{\F}_{\lambda_{k}}\bo_{(0,N)}. 
\end{equation*}

Thus the exact triangle (\ref{comptri}) tells us that ${\E}_{0}{\F}_{N-k}{\F}_{\lambda_1+1}....{\F}_{\lambda_{k}+1}({\spi}^{+})^{-1}\bo_{(0,N)}[-k] \cong {\F}_{\lambda_1}....{\F}_{\lambda_{k}}\bo_{(0,N)}$, which has $\Ff_{\lambda_{1}} \ast ... \ast \Ff_{\lambda_{k}} \bo_{(0,N)} \cong \s_{\blam}\V$ as its FM kernel, and we prove the second case.

\end{proof}

Finally, we have the following result.

\begin{corollary} \label{componentcat}
The generators for the component categories for the two $\SOD$s (\ref{sodk1}) and (\ref{sodk2}) of $\Dd^b(\Gr(k,N))$ are given by  
\begin{align*}
\emph{Im}\Phi_{{\Tt}'\bo_{(k,N-k)}}&= \langle \ \s_{\blam}\V \ \rangle_{0 \leq \lambda_1 \leq N-k-1}, \ 
&\emph{Im}\Phi_{{\Ss}'\bo_{(k,N-k)}}&= \langle \ \s_{\blam}\V \ \rangle_{\lambda_1 =N-k}, \\
\emph{Im}\Phi_{{\Tt}''\bo_{(k,N-k)}}&=\langle \ \s_{\bmu}  (\CC^N/\V)^{\vee}[|\bmu|] \ \rangle_{0 \leq \mu_1 \leq k-1}, \
&\emph{Im}\Phi_{{\Ss}''\bo_{(k,N-k)}}&=\langle \ \s_{\bmu} (\CC^N/\V)^{\vee}[|\bmu|] \ \rangle_{\mu_1=k}. 
\end{align*}
\end{corollary}


\begin{thebibliography}{99}
    \bibitem{AH} M. Abel and M. Hogancamp, \textit{Categorified Young symmetrizers and stable homology of torus links II}. Selecta Math. (N.S.) 23.3 (2017), pp. 1739–1801.

    \bibitem{AL} R. Anno and T. Logvinenko, \textit{Spherical DG-functors}. J. Eur. Math. Soc. (2016), vol. 19 , no. 9, 2577–2656.

    \bibitem{AK1} S. Arkiphov, T. Kanstrup, \textit{Demazure descent and representation of reductive groups}, Algebra, Vol. 2014

    \bibitem{AK2} S. Arkiphov, T. Kanstrup, \textit{Quasi-coherent Hecke category and Demazure descent},Mosc. Math. J. 15 (2015), no. 2, pp. 257-267.
    
    \bibitem{B} A. A. Beilinson, \textit{Coherent sheaves on $\PP^n$ and problems of linear algebra}, Functional Analysis and Its Applications, 12, (1978), no. 3, 68-69.

    \bibitem{BS} D. Bergh, and Olaf M. Schnurer, \textit{Conservative descent for semi-orthogonal decompositions}, Adv. Math. 360 (2020): 106882–39.
    
    \bibitem{BGG} I. N. Bernstein, I. M. Gel’fand, and S. I. Gel’fand, \textit{Schubert cells and cohomology
	of the spaces G/P}, Russian Math. Surveys 28:3 (1973), 1–26.

	\bibitem{BLM} A. A. Beilinson, G. Lusztig, R. MacPherson, \textit{A geometric setting for the quantum deformation of $\GLL_{n}$}, Duke Math. J. 61 (1990), 655-677.

    \bibitem{BN} D. Ben-Zvi, D. Nadler, \textit{Beilinson-Bernstein localization over the Harish-Chandra center}, arXiv:1209.0188v2 
    
    \bibitem{Bon}  A. Bondal, \text{Representations of associative algebras and coherent sheaves}, Izv. Akad. Nauk SSSR Ser. Mat. 53 (1989), no. 1, 25–44.

    \bibitem{BD} M. Boyarchenko, V. Drinfeld, \textit{Character sheaves on unipotent groups in positive characteristic: fundations}, Selecta Math. 20.1 (2014), 125-235.
	

    
    
    

    \bibitem{CR} J. Chuang, R. Rouquier \textit{Derived equivalences for symmetric groups and $\SL_2$-categorification}, Ann. of Math. (2) 167 (2008), no. 1, 245-298.
    
    \bibitem{D1} M. Demazure, \textit{Invariants sym´etriques entiers des groupes de Weyl et torsion},
	Invent. Math. 21 (1973), 287–301.
	
    \bibitem{D2} M. Demazure, \textit{D´esingularisation des vari´et´es de Schubert g´en´eralis´ees.}, Ann. Sci.´ Ecole Norm. Sup. (4) 7 (1974), 53–88.

    
    \bibitem{EH} B. Elias, M. Hogancamp, \textit{Categorical diagonalization}, arXiv:1707.04349v1, 2017
    
    \bibitem{EH1} B. Elias and M. Hogancamp, \textit{Categorical diagonalization of full twists}, 2018. arXiv: 1801.00191.


    \bibitem{FH} W. Fulton and J. Harris. \textit{Representation Theory. A First Course}, Springer-Verlag, New York, (2004).



 
	\bibitem{Hsu} Y. H. Hsu. \textit{A categorical action of the shifted $q=0$ affine algebra}, 	arXiv:2009.03579 [math.RT]
	
	\bibitem{Huy} D. Huybrechts, \textit{Fourier-Mukai transforms in Algebraic Geometry}, Oxford University Press, (2006).

    \bibitem{Ho} M. Hogancamp, \textit{Idempotents in triangulated monoidal cateogries}, arXiv:1703.01001v1, 2017.

    \bibitem{Ho1} M. Hogancamp, \textit{Categorified Young symmetrizers and stable homology of torus links}, Geom. Topol. 22.5 (2018), pp. 2943–3002.

    \bibitem{Ho2} M. Hogancamp, \textit{Constructing categorical idempotents}, arXiv:2002.08905v2

    
    \bibitem{Ka1} M. M. Kapranov. \textit{On the derived categories of coherent sheaves on Grassman manifolds}, Izv. Akad. Nauk SSSR Ser. Mat. 48 (1984), 192-202; English translation in Math. USSR Izv. 24 (1985).
	
	\bibitem{Ka2} M. M. Kapranov. \textit{On the derived categories of coherent sheaves on some homogeneous spaces}, Invent. Math., 92, (1988), no. 3, 479-508.
	


    \bibitem{KL1} M. Khovanov, A. Lauda \textit{A diagrammatic approach to categorification of quantum groups I}, Represent. Theory 13 (2009), 309-347.

    

	
    \bibitem{Ku} A. Kuznetsov, \textit{Base change for semiorthogonal decompositions}, Compos. Math. 147 (2011), no. 3, 852–876.

    \bibitem{Ku'} A. Kuznetsov, \textit{Hochschild homology and semiorthogonal decompositions}, arXiv:0904.4330v1

    \bibitem{Ku1} A. Kuznetsov, \textit{Semiorthogonal decompositions in algebraic geometry}, Proceedings of the International Congress of Mathematicians—Seoul 2014. Vol. II, Kyung Moon Sa, Seoul, 2014, pp. 635–660. 
	
	
	\bibitem{Lu1} G. Lusztig, \textit{Introduction to quantum groups}. Birkhauser, Boston, 1993.

        

    \bibitem{R} R. Rouquier, \textit{2-Kac-Moody algebras}, arXiv:0812.5023 [math.RT]

    

\end{thebibliography}
\end{document}